\documentclass[a4paper,notitlepage,reqno]{article}
\usepackage[english]{babel}
\usepackage{amssymb,amsthm,bm} 
\usepackage{mathtools}
\usepackage{mathrsfs}
\usepackage{enumitem}
\usepackage{authblk}
\usepackage{cite}
\usepackage{xcolor}
\usepackage{tikz}
\usepackage{stackrel}
\usepackage{geometry}
\usepackage{hyperref}

\usepackage{todonotes}

\theoremstyle{plain} 
\newtheorem{theorem}{Theorem}[section]
\newtheorem*{theorem*}{Theorem}
\newtheorem{lemma}{Lemma}[section]
\newtheorem{proposition}{Proposition}[section]
\newtheorem{corollary}{Corollary}[section]
\newtheorem*{corollary*}{Corollary}

\theoremstyle{definition}
\newtheorem{definition}{Definition}

\theoremstyle{remark}
\newtheorem{remark}{Remark}[section]

\numberwithin{equation}{section}

\usepackage{color}
\usepackage[normalem]{ulem}
\definecolor{DarkGreen}{rgb}{0,0.5,0.1} 

\newcommand\soutD{\bgroup\markoverwith
{\textcolor{DarkGreen}{\rule[.5ex]{2pt}{1pt}}}\ULon}

\newcommand{\Hm}[1]{\leavevmode{\marginpar{\tiny%
$\hbox to 0mm{\hspace*{-1.5mm}$\leftarrow$\hss}%
\vcenter{\vrule depth 0.1mm height 0.1mm width \the\marginparwidth}%
\hbox to
0mm{\hss$\rightarrow$\hspace*{-1.5mm}}$\\\relax\raggedright #1}}}

\definecolor{Darkgblue}{rgb}{0.3,0.3,0.5}

\DeclareMathOperator{\Div}{div}

\newcommand{\DeltaH}{\Delta_{\mathbb{H}}}
\newcommand{\nablaH}{\nabla_{\mathbb{H}}}

\title{Unweighted Hardy Inequalities on the Heisenberg Group and in Step-Two Carnot Groups}

\author[1]{Lorenzo D'Arca}
\affil[1]{Department of Mathematics “Guido Castelnuovo”, Sapienza University of Rome,\newline Piazzale Aldo Moro 5, Roma
00185, Italy; lorenzo.darca@uniroma1.it}

\author[2]{Luca Fanelli}
\affil[2]{Ikerbasque, Universidad del País Vasco / Euskal Herriko Unibertsitatea \& BCAM,
\newline Barrio Sarriena s/n, 38940 Leioa, Spain; luca.fanelli@ehu.eus}

\author[3]{Valentina Franceschi}
\affil[3]{Dipartimento di Matematica ``Tullio Levi-Civita'', Universit\`a degli Studi di Padova,\newline via Trieste 63, 35121 Padova, Italy; valentina.franceschi@unipd.it}

\author[4]{Dario Prandi}
\affil[4]{Laboratoire des Signaux et Systèmes, Université Paris-Saclay, CentraleSupélec, CNRS, Gif-sur-Yvette, France}

\begin{document}
\maketitle

\begin{abstract}
We establish unweighted Hardy-type inequalities on step-two Carnot groups with one-dimensional vertical layer, with explicit lower bounds for the optimal Hardy constant. The approach is based on a quantitative integration-by-parts mechanism that replaces the non-horizontal Euler vector field by a suitably constructed horizontal vector field with controlled norm. As applications, we obtain fully explicit bounds in the Heisenberg group for both the Kor\'anyi gauge and the Carnot--Carath\'eodory distance, and we extend the results to non-isotropic step-two structures through a generalized Kor\'anyi-type homogeneous norm.
\medskip

\noindent\textbf{Keywords:} Hardy inequalities; Carnot groups; Heisenberg group; sub-Laplacian; homogeneous norms; Carnot--Carath\'eodory distance; inverse-square potentials.

\noindent\textbf{2020 Mathematics Subject Classification:} Primary 26D10; Secondary 35J70, 35P15, 43A80.
\end{abstract}

\section{Introduction}

The Hardy inequality is an important tool in modern analysis, linking geometry with potential and spectral theory. 
In its simplest Euclidean form it reads 
\begin{equation}
    \label{eq:hardy-euclidean}
    \int_{\mathbb R^N} |\nabla u|^2 \,dx \ge C_N \int_{\mathbb R^N}\frac{|u|^2}{|x|^2}\, dx,
    \qquad \forall u\in C^\infty_c(\mathbb R^N\setminus\{0\}),
\end{equation}
where the constant $C_N = (N-2)^2/4$ quantifies how the gradient controls boundary singularities.  

When $\mathbb R^N\setminus\{0\}$ is replaced by an open set $\Omega$ and $|x|$ by the distance from its boundary $d(x,\partial\Omega)$, it is known that the corresponding optimal constant $C(\Omega)$ is connected with appropriate notions of dimension of $\partial\Omega$, e.g., Minkowski and Assouad dimensions \cite{daviesReview1999}. 
In particular, its positivity is guaranteed under sufficient thickness or thinness of $\partial \Omega$ \cite{lehrbackHardy2017}.

On the other hand, $C_N=C(\mathbb R^N\setminus \{0\})$ identifies the critical value of the coupling parameter $\lambda$ for the operator 
\begin{equation}
    P_\lambda = -\Delta - \frac{\lambda}{|x|^2} \qquad \text{ on } L^2(\mathbb{R}^N\setminus\{0\}).
\end{equation}
Indeed, $C_N$ marks the precise threshold between the subcritical regime for $\lambda<C_N$, where the operator remains positive and the associated energy functional is coercive, and the supercritical regime for $\lambda>C_N$, where coercivity fails and solutions may blow up near the singular set.
At the critical level, the operator becomes degenerately coercive: the Hardy inequality is still valid, but it cannot be improved, and the corresponding ``ground state'' 
profile saturates the inequality, \cite{CKN}. 

In the Heisenberg group (see Section~\ref{sec:heisenberg}) Hardy inequalities have been studied since \cite{GL}, where the authors prove the following sharp Hardy inequality
\begin{equation}
    \label{eq:GL-hardy}
    \int_{\mathbb H^n} |\nablaH u|^2 \,dx \ge \left(
    \frac{Q-2}2 
    \right)^2 \int_{\mathbb{H}^n}\frac{|u|^2}{\rho^2}|\nablaH \rho|^2\, dx,
    \qquad \forall u\in C^\infty_c(\mathbb{H}^n\setminus\{0\}),
\end{equation}
where $\nablaH$ denotes the horizontal gradient, $\rho$ is the Koranyi norm in $\mathbb{H}^n$, and $Q=2n+2$ is the homogeneous dimension of $\mathbb{H}^n$. 
The original motivation of the authors was the unique continuation property for the Schrödinger operator $-\DeltaH+V$, where $\DeltaH$ is the Heisenberg sublaplacian and $V$ is a potential. In particular, owing to \eqref{eq:GL-hardy} they derive these properties for the operator
\begin{equation}
    \label{eq:op-koranyi}
    -\DeltaH +\lambda \frac{|\nablaH \rho|^2}{\rho^2}, \qquad \lambda \in\mathbb{R},
\end{equation}
which is a weighted counterpart to $P_\lambda$ in the Heisenberg group.
These facts have then been generalized to the Carnot group setting \cite{dambrosioHardy2005}.

The weight $|\nablaH \rho|$ appearing in \eqref{eq:GL-hardy} vanishes along the vertical direction and, accordingly, the potential appearing in \eqref{eq:op-koranyi} is ineffective along this direction. 
This fact prevented, e.g., the authors of \cite{adamiPoint2021} to employ techniques based on the Hardy inequality \eqref{eq:GL-hardy} to prove the self-adjointness of the Heisenberg sublaplacian.
This motivates the search pursued in this work for unweighted Hardy inequalities. 

The above inequality and its generalizations to different settings hold mainly due to the relationship between $\rho$ and the fundamental solution of the sublaplacian \cite{follandFundamental1973}.
This fact constitutes a major obstruction to establishing Hardy-type inequalities where the Korányi norm $\rho$ is replaced by the more geometrically intrinsic Carnot–Carathéodory distance from the origin, $\delta_{cc}$. 
In this case, the natural analogue of \eqref{eq:GL-hardy} is unweighted, owing to the fact that $|\nabla_H \delta_{cc}| = 1$ a.e.. Namely, 
\begin{equation}
    \label{eq:unweighted-hardy-cc}
    \int_{\mathbb H^n} |\nablaH u|^2 \,dx \ge c \int_{\mathbb{H}^n}\frac{|u|^2}{\delta_{cc}^2}\, dx,
    \qquad \forall u\in C^\infty_c(\mathbb{H}^n\setminus\{0\}).
\end{equation}
Such inequality for some non-sharp $c>0$ has been proven in \cite{BCX,BCG}.
This inequality is related to the following operator, to be compared with \eqref{eq:op-koranyi},
\begin{equation}
    \label{eq:inverse-square-heis}
    -\DeltaH + \frac{\lambda}{\delta_{cc}^2}.
\end{equation}
In the Euclidean setting, the Hardy inequality \eqref{eq:hardy-euclidean} follows from its radial version
\begin{equation}
    \label{eq:hardy-euclidean-radial}
\int_{\mathbb{R}^N} \left| \left\langle\nabla u,\frac{x}{|x|} \right\rangle\right|^2 dx \geq \left( \frac{N - 2}{2} \right)^2 \int_{\mathbb{R}^N} \frac{|u|^2}{|x|^2} dx, \quad \forall u \in C_c^\infty(\mathbb{R}^N \setminus \{0\}).
\end{equation}
Indeed, minimizing sequences for the above can be chosen to be radial so that the only non-zero component of their gradient is the radial one. In particular, such sequences are then minimizing also for \eqref{eq:hardy-euclidean}. 
We stress that this is connected with the existence of rearrangement inequalities, as the Pólya-Szegő inequality, which are not available in the Heisenberg setting \cite{montiRearrangements2013,manfrediRearrangements2019}. 
In fact, as shown in \cite{FP}, the radial Hardy inequality obtained by projecting the horizontal gradient along the direction $\nabla\delta_{cc}$ is trivial (i.e., holds with constant $0$) and, in particular, does not imply the complete Hardy inequality \eqref{eq:unweighted-hardy-cc} with $c>0$.

In this work, we start by taking a closer look to the notion of radial derivative.
In the Euclidean setting, the operator \( \partial_r =\left\langle\frac{x}{|x|} , \nabla \right\rangle\) can be viewed either as (i) the projection of the gradient in the direction of the distance from the origin, i.e., \(\partial_r=\left\langle\nabla |x|, \nabla\right\rangle\), or (ii) as \(\partial_r=\frac{\langle x, \nabla\rangle}{|x|}\), where \(\langle x , \nabla \rangle\) is the infinitesimal generator of Euclidean dilations, known as Euler vector field. 
%

In \(\mathbb{H}^n\), these two points of view no longer coincide nor for the Carnot-Carthéodory distance \cite{FP}, nor for the Korányi norm $\rho$.
%
Focussing on the latter, the first approach leads to the operator
\[
\left\langle\frac{\nabla_\mathbb{H} \rho}{|\nabla_\mathbb{H} \rho|}, \nabla_\mathbb{H}\right\rangle.
\]
Albeit in a different language, this is the approach taken in \cite{GL}, leading to 
the introduction of the weight \(|\nabla_\mathbb{H} \rho|^2\) in the right-hand side of Hardy inequality, as in \eqref{eq:GL-hardy}. 
On the other hand, arguing as in the second approach leads to consider the Euler vector field associated with the Heisenberg homogeneous structure. 
This is obtained as the generator of the non-isotropic dilations  $\delta_\lambda(x,y,t)=(\lambda x, \lambda y, \lambda^2t)$, where $(x,y,t)\in \mathbb{R}^{n}\times \mathbb{R}^n\times \mathbb{R}\simeq\mathbb{H}^n$, and reads
\[
\mathcal{E} = \sum_{i=1}^n x_i \partial_{x_i} + y_i \partial_{y_i} + 2t \partial_t,
\]
%
In this case, the natural definition of radial derivative is \(\partial_\rho =\rho^{-1} \mathcal{E}\), and, as observed in \cite{RS}, the corresponding Hardy inequality takes the following unweighted form:
\[
\int_{\mathbb{H}^n} |\partial_\rho u|^2 \, dzdt \geq \left( \frac{Q - 2}{2} \right)^2 \int_{\mathbb{H}^n} \frac{|u|^2}{\rho^2} \, dzdt.
\]
However, since the Euler field \(\mathcal{E}\) is not horizontal, this inequality does not yield any control on \(|\nabla_\mathbb{H} u|\). 
This issue represents one of the main difficulties when trying to derive unweighted Hardy inequalities in sub-Riemannian settings.
\medskip

So far inequalities as \eqref{eq:unweighted-hardy-cc} have been proved only with non-explicit constants \cite{BCG}, and the sharp value of $c$ remains unknown. The best available estimate at the time of this writing is $0<c<(Q-2)^2/4$, proved in \cite{FP}.
Moreover, due to the equivalence of norms in the finite dimensional setting, establishing an inequality as \eqref{eq:unweighted-hardy-cc} for some positive constant $c>0$ automatically yields that the same holds for any other norm, albeit the information on the exact value of the constant is lost. 

In this paper, we study unweighted Hardy inequalities as \eqref{eq:unweighted-hardy-cc}, on step-two Carnot groups with one-dimensional vertical layer, for various choices of a homogeneous norm $d$.
More generally, we are interested in establishing inequalities of the following type
\begin{equation}
\label{eq:general-p-theta-hardy}
\int_{\mathbb{H}^n} \frac{|\nabla_\mathbb{H} u|^p}{d^{p(\theta - 1)}} \, dzdt \geq c \int_{\mathbb{H}^n} \frac{|u|^p}{d^{p \theta}} \, dzdt,
\end{equation}
where \( p \geq 2 \) and \( \theta \in \mathbb{R} \) is a fixed parameter, which are associated with nonhomogeneous generalizations of the operator \eqref{eq:inverse-square-heis} to the nonlinear $L^p$ setting. 

Results of this type are currently available only with implicit Hardy constants. 
As mentioned above, in the Heisenberg group, they were first obtained by Bahouri–Chemin–Xu and Bahouri–Chemin–Gallagher \cite{BCX,BCG}. These results were later extended to general step-two Carnot groups by Vigneron \cite{V}. All these works cover the case $p=2$ and $\theta=1$. For $\theta\neq1$, the available inequalities involve different norms on the left-hand side and therefore do not correspond to the weighted formulation considered here.

The aim of the present work is to provide a quantitative refinement of these results, establishing explicit lower bounds for the optimal Hardy constant.

\subsection{Techniques}
Here lies the core idea of our work:  Although \(\mathcal{E}\) is not horizontal, we show that it is possible, via an integration by parts argument, to find a horizontal vector field \(Z\), bounded in norm, such that the following equality holds,
\begin{equation}
\label{eq: E e Z in L2}
	\int_{\mathbb{H}^n} \frac{u \mathcal{E} u}{\rho^2} \, dzdt = \int_{\mathbb{H}^n} \frac{u \,\langle\nabla_\mathbb{H} u , Z\rangle}{\rho} \, dzdt.
\end{equation}
Starting from this identity, an unweighted Hardy inequality for $\rho$ follows from the computation
\[
\begin{split}
0&\leq \int_{\mathbb{H}^n} \left| \left\langle\nabla_\mathbb{H} u , Z\right\rangle + \frac{Q - 2}{2} \frac{u}{\rho} \right|^2 \, dzdt\\
&=\int_{\mathbb{H}^n} |\left\langle\nabla_{\mathbb{H}} u , Z\right\rangle|^2 \, dz dt + \left( \frac{Q - 2}{2} \right)^2 \int_{\mathbb{H}^n} \frac{|u|^2}{\rho^2} dz dt + (Q - 2) \int_{\mathbb{H}^n} \frac{u \mathcal{E} u}{\rho^2} dz dt.
\end{split}
\]
Using the adjoint identity 
\(\mathcal{E}^* = -Q - \mathcal{E}\), we arrive at
\begin{equation}
\int_{\mathbb{H}^n} |\nabla_\mathbb{H} u|^2 \, dzdt \geq \frac{1}{\sup_{\mathbb{H}^n} |Z|^2} \left( \frac{Q - 2}{2} \right)^2 \int_{\mathbb{H}^n} \frac{|u|^2}{\rho^2} \, dzdt.
\end{equation}
As a consequence, upper bounds for $|Z|$ yield lower bounds of the Hardy constant.
%
%
%


The idea of replacing the non-horizontal Euler vector field by a horizontal one through an integration-by-parts argument was already explored by Vigneron in \cite{V}. His analysis, however, remained qualitative, the emphasis was on the existence of Hardy-type inequalities rather than on determining the associated constants.
In contrast, the present work adopts a fully quantitative perspective, keeping precise track of all constants arising in the construction. Our results are established in the nonlinear $L^p$ setting, within the general framework of step-two Carnot groups whose vertical layer is one-dimensional.

In the final part of the paper, Section~\ref{sec:heisenberg}, we specialize to the Heisenberg group, the most prominent example in this class, where the geometry allows for completely explicit computations. We also briefly address Carnot groups with higher dimensional vertical layers, such as direct products of Heisenberg groups or more general step two structures, in which the same method applies in principle, although the resulting computations become substantially more intricate.

\subsection{Main results}\label{sec:mainresults}

Let us consider \(\mathcal{G} = (\mathbb{R}^m \times \mathbb{R}, \circ)\), $m\in\mathbb N$, to be a step $2$ Carnot group with $1$-dimensional vertical direction.
Namely, there exists a skew-symmetric \( m \times m \) matrix \( B \), such that the group law on \( \mathbb{R}^m \times \mathbb{R} \) is defined by
\[
(z, t) \circ (\eta, \tau) = \left(z + \eta,\ t + \tau + \frac{1}{2} \langle Bz, \eta \rangle\right)
\qquad (z,t),(\eta,\tau)\in \mathbb{R}^m\times \mathbb{R},
\]
where $\langle\cdot,\cdot\rangle$ denotes the Euclidean scalar product on $\mathbb{R}^m$.
The associated family of dilations is given by \(\delta_\gamma(z, t) = (\gamma z,\, \gamma^2 t)\), which defines a one-parameter group of automorphisms of the group, and an orthonormal basis of horizontal vector fields is given by 
\begin{equation}
    \label{eq:ortho-1}
    X_i = \partial_{z_i}+\frac{1}2 (Bz)_i \partial_t, \qquad i=1,\ldots, m,
\end{equation}
%
Expressing horizontal vectors in this basis, the scalar product on the horizontal layer of $\mathcal G$ reduces to the standard Euclidean product on $\mathbb{R}^m$.
The horizontal gradient and the sub-Laplacian of a smooth function $u$ admit the explicit expression
%
%
\begin{equation}
    \label{eq: sub-Laplaciano su G}
    \nabla_{\mathcal{G}} u =  (X_1 u, \ldots , X_m u)
    \qquad\text{and}\qquad
    \Delta_{\mathcal{G}}u = \Delta_z u + \frac{1}{4} |Bz|^2 \, \partial_t^2 u + \langle Bz, \nabla_z u \rangle \, \partial_t u,
\end{equation}
where \(\Delta_z\) and \(\nabla_z\) denote the usual Laplacian and gradient on \(\mathbb{R}^m\), respectively. 
Observe that the commutators satisfy \([X_j, X_i] = B_{ij} \, \partial_t\), and $[\partial_t , X_i]=0$.
Since \(B\neq 0\), the Lie algebra generated by \(\{X_1, \dots, X_m\}\) has full rank, i.e., this family satisfy the Hörmander condition. This implies that $\Delta_{\mathcal G}$ is an hypoelliptic operator \cite{hormanderHypoelliptic1967}.

We will assume the following
\begin{equation}
    \tag{H0}
    \label{eq:H0}
    \text{The kernel of $B$ is trivial.}
\end{equation}
%
The prototypical example of group satisfying these assumptions is the Heisenberg group \(\mathbb{H}^n\), which corresponds to \(m = 2n\) and
\[
B = \begin{pmatrix} 0 & 4 I_n \\ -4 I_n & 0 \end{pmatrix}.
\]
For further background, we refer the reader to \cite{BLU}.

\begin{definition}
    A continuous function \(d: \mathbb{R}^{m+1} \to [0, +\infty)\) is positive homogeneous on $\mathcal G$ if
\begin{itemize}
\item \(d(\delta_\lambda(z, t)) = \lambda\, d(z, t)\) for every \(\lambda > 0\) and \((z, t) \in \mathbb{R}^m \times \mathbb{R}\),
\item \(d(z, t) > 0\) for all \((z, t) \neq (0, 0)\).
\end{itemize}
Moreover, $d$ is \emph{regular} if
\begin{itemize}
\item[(d.1)]\label{condizione i} \(d \in C^\infty(\mathcal{G} \setminus L)\), where \(L = \{(z, t) \in \mathbb{R}^{m} \times \mathbb{R} : |z| = 0\}\) denotes the center of the group

\item[(d.2)] the vector field \(\nabla_{\mathcal{G}} d\) and the derivative \(\partial_t d\) are locally bounded away from the origin, namely
\begin{equation*} 
\|\nabla_{\mathcal{G}} d\|_{L^\infty_{\mathrm{loc}}(\mathcal{G} \setminus \{0\})} < +\infty, 
\qquad 
\|\partial_t d\|_{L^\infty_{\mathrm{loc}}(\mathcal{G} \setminus \{0\})} < +\infty;
\end{equation*}
%

\item[(d.3)] \(\text{for a.e. } t \in \mathbb{R}\) the following holds
%
\begin{equation*} 
\label{eq: d perp lungo z}
\lim_{|z| \to 0} {\left\langle \frac{z}{|z|},B^{-1}\nabla_{\mathcal G}d \right\rangle} = 0 \quad \text{for a.e. } t \in \mathbb{R},
\end{equation*}
%
\end{itemize}
\end{definition}

\begin{remark}
\label{rmk:hom-norm}
    The horizontal Lipschitz continuity of $d$ outside the origin is automatically guaranteed for homogeneous norms, i.e., positive homogeneous functions that satisfy the triangle inequality
    \begin{equation}
        d((z,t)\circ(\eta,\tau)) \le d(z,t)+d(\eta,\tau), \quad \forall (z,t),(\eta,\tau)\in\mathcal G.
    \end{equation}
    The full condition (d.2) is satisfied, for instance, by any positive homogeneous function that is Euclidean Lipschitz outside the origin, as the Korányi or the Carnot-Carthéodory norms in the Heisenberg group. This follows, e.g., observing that $\nabla_{\mathcal G} d$ and $d\partial_t d$ are both homogeneous of degree $0$.
\end{remark}

Our main result is then the following.

\begin{theorem}
\label{thm:-Hardy-lungo-Z_d}
Let \( p \geq 2 \), \( \theta \in \mathbb{R} \), and let \( d \) be a regular positive homogeneous function on \( \mathcal{G} \). Then, for every \( u \in C_c^\infty(\mathcal{G} \setminus \{0\}) \), the following Hardy-type inequality holds
\begin{equation}
\label{eq:-HardyZ_d}
\int_{\mathcal{G}} \frac{|\left\langle\nabla_{\mathcal{G}} u , Z_d\right\rangle|^p}{d^{p(\theta - 1)}} \, dz\,dt 
\geq \left| \frac{Q - p\theta}{p} \right|^p \int_{\mathcal{G}} \frac{|u|^p}{d^{p\theta}} \, dz\,dt.
\end{equation}
Here, the vector field $Z_d$ is defined by
\begin{equation}
\label{eq:-definizione-di-Z_d}
Z_d = 
\frac{m+2}{m} \,\, \frac{z}{d} 
\ -\ 
\frac{4p\theta}{m} \,\, \frac{t}{d^2} 
B^{-1}\nabla_{\mathcal G} d.
\end{equation}
If, additionally, \( d \) is such that $\langle z,B^{-1}\nabla_z d\rangle=0$ for all $z\in \mathcal G\setminus\{0\}$, 
then inequality \eqref{eq:-HardyZ_d} is sharp, and equality is attained by 
\[
u(z, t) = \left( \frac{|t|}{|z|^2} \right)^{\frac{Q - 2}{2p}}.
\]
\end{theorem}

\begin{remark}\label{rmk:-ipotesi-planar-rotations}
The assumption \(\langle z , B^{-1}\nabla_z d\rangle = 0\) is satisfied precisely when \(d\) is constant along the integral curves of
\[
    v(z) = (B^{-1})^{T} z .
\]
Since \(B\) is skew symmetric, the flow generated by \(v\) consists of planar rotations in the \(z\)-variables.  
In the Heisenberg group, this corresponds to simultaneous rotations in each pair \((x_j,y_j)\), as in the case of the Korányi and Carnot–Carathéodory norms.
\end{remark}

We now turn to providing estimates on the Hardy constant, which is defined by
\begin{equation}
    \label{eq:cdp}
        c(d,p,\theta) := \sup \left\{ c\ge 0 \mid  
        \int_{\mathcal{G}} \frac{|\nabla_{\mathcal{G}} u|^p}{d^{p(\theta - 1)}} \, dz\,dt 
        \geq c \int_{\mathcal{G}} \frac{|u|^p}{d^{p\theta}} \, dz\,dt,\text{ for any } u \in C_c^\infty(\mathcal{G} \setminus \{0\})\right\},
\end{equation}
for a positive homogeneous function $d$, $p\ge 2$, and $\theta\in \mathbb{R}$.
Indeed, we have the following immediate corollary.

\begin{corollary}\label{corr:lower-bound}
    Let $p\ge 2$, $\theta\in\mathbb{R}$, and let $d$ be a  regular positive homogeneous function on $\mathcal{G}$. 
    Then,
    \begin{equation}
        c(d,p,\theta) \ge \frac{1}{\sup_{\mathcal G}|Z_d|^p} \left| \frac{Q - p\theta}{p} \right|^p.
    \end{equation}
\end{corollary}

Specifying the above for suitable choices of $d$ we obtain explicit lower bounds of the Hardy constant \eqref{eq:cdp}.
Unsurprisingly, the case most amenable to these direct computation is the Heisenberg group (we refer to Section~\ref{sec:heisenberg} for precise definitions), where we can establish the following.

\begin{theorem}\label{thm:constant-hardy-heisnberg}
    Let $p\ge 2$, $\theta\in\mathbb{R}$ and consider $\mathcal G= \mathbb H^n$.
    Then, we have the following
    \begin{enumerate}
        \item For the Korányi norm $\rho$, it holds
        \begin{equation}
            \label{eq:cdp-koranyi}
            c(\rho,p,\theta) \ge 
            \begin{cases}
                 \displaystyle\left|\frac{Q - p \theta}{p}\right|^p\left|\frac{Q-2}{Q}\right|^p & \text{ if } p\theta\in \left[(1 - \sqrt{\frac{3}{2}})Q , (1 + \sqrt{\frac{3}{2}})Q\right],\\[2em]
                \displaystyle\left(\frac{3}{2}\right)^{\frac{p}{2}}  \frac{\left(3 p \theta (p \theta - 2Q)\right)^{\frac{p}{4}}}{|p \theta - Q|^{\frac{p}{2}}}\left|\frac{Q-2}{p}\right|^p & \text{otherwise.}
            \end{cases}
        \end{equation}
        \item For the Carnot-Carathéodory norm $\delta_{cc}$, it holds
        \begin{equation}
            \label{eq:cdp-carnot-caratheodory}
            c(\delta_{cc},p,\theta) \ge
            \begin{cases}
                \displaystyle\left(\frac{Q - 2}{Q}\right)^p \left|\frac{Q - p \theta}{p}\right|^p& \text{ if } \theta \geq 0 \text{ and } Q \geq \dfrac{4 p \theta}{12 - \pi^2}\\[2em]
                \displaystyle\frac{1}{\|g\|_\infty^{p/2}} \left|\frac{Q - p \theta}{p}\right|^p,&\text{otherwise.}
            \end{cases}
        \end{equation}
        Here, $g:\nu\in [-2\pi,2\pi]\mapsto g(\nu)\in (0,+\infty)$ is defined  by 
        \begin{equation}
            g(\nu) = 2 \left(\frac{Q}{Q-2}\right)^2 \frac{1-\cos\nu}{\nu^2} + \left(\frac{2 p \theta}{Q-2}\right)^2 \left(\frac{\nu-\sin\nu}{\nu^2}\right)^2 - \frac{4 p \theta Q}{(Q-2)^2} \frac{\nu-\sin\nu}{\nu^2} \frac{1-\cos\nu}{\nu}.
        \end{equation}
    \end{enumerate}
\end{theorem}

In particular, for $p=2$ and $\theta=1$ this provides the first explicit positive lower bound for the optimal Hardy constant, thereby improving the estimate $0<c<(Q-2)^2/4$ obtained in \cite{FP}.

Concerning more general step $2$ Carnot groups satisfying \eqref{eq:H0}, we can establish the explicit lower bounds for the Hardy constant associated with the following generalization of the Korányi norm:
\begin{equation}
    \label{eq:koranyi-generalized}
    \rho_B(z,t) = \left( |z|_B^4 + |t|^2 \right)^{\frac{1}{4}}, \qquad |z|_B = \frac12\sqrt{\langle z, (-B^2)^{1/2} z \rangle}.
\end{equation}
Here, the matrix $(-B^2)^{1/2}$ is well-defined and positive definite, since $B$ is skew-symmetric and invertible, and the norm $|\cdot|_B$ is usually referred to as the symplectic norm associated with $B$.
Observe that when $\mathcal G=\mathbb H^n$, then $|z|_B=|z|$
and $\rho_B$ reduces to  the standard Korányi norm $\rho$.

We have the following result.

\begin{theorem}
    \label{thm:constant-hardy-carnot-generalized}
    Consider $\mathcal G$ a step $2$ Carnot group satisfying \eqref{eq:H0}.
    For any $p\ge 2$ and $\theta\in\mathbb{R}$, we have the following:
        \begin{equation}
        \label{eq:cdp-koranyi-generalized}
        c(\rho,p,\theta) \ge 
        \begin{cases}
             \displaystyle\left(\frac{\lambda_{\min}(B)}4\right)^{\frac p2}\left|\frac{Q - p \theta}{p}\right|^p\left|\frac{Q-2}{Q}\right|^p & \text{ if } p\theta\in \left[(1 - \sqrt{\frac{3}{2}})Q , (1 + \sqrt{\frac{3}{2}})Q\right],\\[2em]
            \displaystyle \left(\frac{\lambda_{\min}(B)}4\right)^{\frac p2}
            \left(\frac{3}{2}\right)^{\frac{p}{2}}\,\displaystyle  \frac{\left(3 p \theta (p \theta - 2Q)\right)^{\frac{p}{4}}}{|p \theta - Q|^{\frac{p}{2}}}\left|\frac{Q-2}{p}\right|^p  & \text{otherwise.}
    \end{cases}
    \end{equation}
    Here and in the sequel, we set $\lambda_{\min}(B) := \min \sigma\bigl(( -B^2)^{1/2}\bigr)$, that is, the smallest eigenvalue of the symmetric positive definite matrix $( -B^2)^{1/2}$.
    In particular, in the case $p=2$ and $\theta=1$, we have 
    \begin{equation}
        c(\rho_B, 2,1) \ge \displaystyle \frac{\lambda_{\min}(B)}4\,\frac{(Q - 2)^4}{4Q^2}.
    \end{equation} 
\end{theorem}

We conclude this work by considering step $2$ groups with more than one vertical direction.
We first obtain an analogue of Theorem~\ref{thm:-Hardy-lungo-Z_d} for the case of $N$ products of
$\mathbb{H}^n$ (see Theorem~\ref{thm:-ProH-Hardy-lungo-Z}), from which the following corollary follows.
We then extend the construction to general step $2$ Carnot groups with several vertical
directions in Theorem~\ref{thm:caso-carnot-generale}.

\begin{corollary}
Let \( p \geq 2 \), \( \theta \geq 0 \), and let \( \rho(z,t) \) be the Korányi norm on \( \mathcal{G} = (\mathbb{H}^n)^N \). Then, if
\[
n \geq \frac{1}{4}(p\theta - 4),
\]
the following Hardy-type inequality holds for all \( u \in C_c^\infty(\mathcal{G} \setminus \{0\}) \)
\begin{equation}
\label{eq:-ProH-hardy-non-pesata}
\int_{\mathcal{G}} \frac{|\nabla_{\mathcal{G}} u|^p}{\rho^{p(\theta-1)}} \, dz \, dt 
\geq \left( \frac{n}{n+1} \right)^p \left| \frac{Q - p\theta}{p} \right|^p \int_{\mathcal{G}} \frac{|u|^p}{\rho^{p\theta}} \, dz \, dt.
\end{equation}
\end{corollary}

\subsection{Structure of the paper}
Section~\ref{sec:preliminaries} introduces the notation and collects some preliminary results.
Section~\ref{sec:proof-theorem} first establishes the identity relating the vector fields \( Z_d \) and \( \mathcal{E} \) (Proposition~\ref{prop:-identità-tra-Z-e-E}), which plays a fundamental role in the argument, and then proves Theorem~\ref{thm:-Hardy-lungo-Z_d}.

Section~\ref{sec:unweighted-hardy-inequality} is devoted to explicit
unweighted Hardy inequalities.
We begin with the Heisenberg group, comparing the Korányi and
Carnot Carathéodory norms (Section~\ref{sec:heisenberg}), and then
consider a non-isotropic example
(Section~\ref{sec:non-isotropic-case}).
We next turn to groups with more than one vertical direction, treating
the case of $N$ products of $\mathbb{H}^n$
(Section~\ref{sec:product-heisnebrg}) and subsequently the general
step two setting.

\section{Preliminaries}\label{sec:preliminaries}

Let us consider \(\mathcal{G} = (\mathbb{R}^m \times \mathbb{R}, \circ)\), $m\in\mathbb N$, to be a step $2$ Carnot group with $1$-dimensional vertical direction satisfying assumption \eqref{eq:H0}.

We start by considering an appropriate change of coordinates, which simplifies computations in the following.
Since $B$ is skew-symmetric and invertible, the group $\mathcal G$ is even-dimensional (i.e., $m=2n$ for $n\in \mathbb{N}$) and there exists $\lambda_1,\ldots,\lambda_n\in\mathbb{R}\setminus\{0\}$ and an orthogonal matrix $Q$ such that
\begin{equation}\label{eq:-Definizione-matrice-B-diagonalizzata}
    B = Q\Lambda Q^\top, 
    \qquad\text{where}\qquad
    \Lambda = 
\begin{pmatrix}
\begin{matrix}
0 & \lambda_1 \\
-\lambda_1 & 0
\end{matrix} & & & \\
& \ddots & & \\
& & 
\begin{matrix}
0 & \lambda_n \\
-\lambda_n & 0
\end{matrix} 
\end{pmatrix}
\in \mathbb{R}^{m \times m}.
\end{equation}
Up to an orthogonal change of coordinates, we may and shall assume that, for all  $i$,
\[
\lambda_i>0 .
\]
Then, performing the orthogonal change of variables $z\mapsto Q^\top z$, we reduce to the case $B=\Lambda$.


In these coordinates, the horizontal layer is spanned by the left-invariant vector fields
\[
X_{2i-1}=\partial_{z_{2i-1}}+\frac{\lambda_i}{2}\, z_{2i}\,\partial_t,
\qquad
X_{2i}=\partial_{z_{2i}}-\frac{\lambda_i}{2}\, z_{2i-1}\,\partial_t,
\qquad i=1,\dots,n,
\]
and the vertical direction is \(T=\partial_t\).
The only nontrivial commutators are
\[
[X_{2i},X_{2i-1}]=\lambda_i\,T,\qquad i=1,\dots,n.
\]

For later use we set
\begin{gather}
    \label{eq:Definizione-nabla-2i-1-2i}
\nabla_{2i-1,2i}u=(0,\dots,0,X_{2i-1}u,X_{2i}u,0,\dots,0),\\
\nabla_{2i-1,2i}^\perp u=(0,\dots,0,-X_{2i}u,X_{2i-1}u,0,\dots,0),
\end{gather}
and denote
\[
\mathbf z^{(i)}=(0,\dots,0,z_{2i-1},z_{2i},0,\dots,0),\qquad
\mathbf z^{(i),\perp}=(0,\dots,0,-z_{2i},z_{2i-1},0,\dots,0).
\]


The next lemmas collect a few auxiliary identities that will be used instrumental to the proof of Theorem~\ref{thm:-Hardy-lungo-Z_d}.

\begin{lemma}\label{lemma:-ricostruzione_d}
Let $d$ be a regular positive homogeneous function on $\mathcal G$ and assume
\[
\langle z , B^{-1}\nabla_z d\rangle = 0 .
\]
Then, on $\mathcal G\setminus L$,
\[
4 \,\frac{t}{|z|^2} \left\langle B^{-1}\nabla_{\mathcal{G}} d , 
z\right\rangle
\;+\; 
\langle z , \nabla_{\mathcal G} d \rangle
= d .
\]
\end{lemma}

\begin{proof}
The homogeneity of $d$ yields $\mathcal{E}d = d$, i.e.,
\begin{equation}
\langle z,\nabla_{\mathcal G} d\rangle + 2t\,\partial_t d = d.
\label{eq:homogeneity-d}
\end{equation}
Computations via the explicit form of the vector fields give
\[
\begin{split}
\left\langle B^{-1}\nabla_{\mathcal{G}} d , z \right\rangle
&=\sum_{i=1}^n \frac{1}{\lambda_i}\, \langle \mathbf{z}^{(i)} ,\nabla_{2i-1,2i}^\perp d\rangle\\
&= \sum_{i=1}^n \frac{1}{\lambda_i}\bigl( z_{2i} X_{2i-1} d - z_{2i-1} X_{2i} d \bigr)\\
&=\sum_{i=1}^n \frac{1}{\lambda_i}
\left( 
z_{2i}\,\frac{\partial d}{\partial z_{2i-1}}
- z_{2i-1}\,\frac{\partial d}{\partial z_{2i}}
\right)
+ 
\frac12 
\sum_{i=1}^n (z_{2i}^2+z_{2i-1}^2)\,\partial_t d\\
&= \langle z, B^{-1}\nabla_z d \rangle  + \frac{|z|^2}{2}\,\partial_t d.
\end{split}
\]
The first term on the r.h.s. vanishes by the assumption 
$\langle z , B^{-1}\nabla_z d\rangle = 0$. Hence, multiplying the above by $4t/|z|^2$ gives
\begin{equation}
4\,\frac{t}{|z|^2}\, \left\langle
B^{-1}\nabla_{\mathcal{G}} d , 
z \right\rangle
=
2t\,\partial_t d.
\label{eq:lemma2-vertical}
\end{equation}
Substituting \eqref{eq:lemma2-vertical} into \eqref{eq:homogeneity-d}
yields the identity.
\end{proof}

\begin{lemma}\label{lemma:-div^perp_casoRegolare}
Let $\Omega \subseteq \mathcal G$ be open and let $\varphi \in C^\infty(\Omega)$.  
Then, for every $i=1,\dots,n$,
\[
\Div_{\mathcal G}\bigl(\nabla_{2i-1,2i}^{\perp}\varphi\bigr)
= \lambda_i \,\partial_t \varphi .
\]
\end{lemma}

\begin{proof}
Using the definition of the horizontal divergence,
\[
\Div_{\mathcal G}\bigl(\nabla_{2i-1,2i}^{\perp}\varphi\bigr)
= X_{2i-1}(-X_{2i}\varphi)+X_{2i}(X_{2i-1}\varphi)
= [X_{2i},X_{2i-1}]\,\varphi .
\]
Since $[X_{2i},X_{2i-1}] = \lambda_i\,\partial_t$, the conclusion follows.
\end{proof}




\begin{lemma}\label{lemma:div-dist}
Let $d$ be a regular positive homogeneous function on $\mathcal G$.  
Then, in the sense of distributions on $\mathcal G\setminus\{0\}$, the following identities hold:

\begin{itemize}
\item[\textnormal{(i)}]
\(\displaystyle
\Div_{\mathcal G}\!\left(
\frac{t}{d^{p\theta+1}}\,
B^{-1}\nabla_{\mathcal G}d
\right)
=
-\frac12\,\frac{\langle z,\nabla_{\mathcal G} d\rangle}{d^{p\theta+1}}
\;+\;
n\,\frac{t}{d^{p\theta+1}}\,\partial_t d ,
\)

\item[\textnormal{(ii)}]
\(\displaystyle
\Div_{\mathcal G}\!\left(\frac{z}{d^{p\theta}}\right)
=
\frac{2n}{d^{p\theta}}
-
\frac{p\theta}{d^{p\theta+1}}\, \, \langle z,\nabla_{\mathcal G} d \rangle.
\)
\end{itemize}
\end{lemma}

\begin{proof}
A direct computation gives
\[
\nabla_{2i-1,2i}\!\left(\frac{t}{d^{p\theta+1}}\right)
=
-\frac{\lambda_i}{2}\,\mathbf z^{(i),\perp}
-
(p\theta+1)\,\frac{t}{d^{p\theta+2}}\,
\nabla_{2i-1,2i} d .
\]
    Hence, using Lemma~\ref{lemma:-div^perp_casoRegolare}, we obtain that
    \begin{equation}
        \Div_{\mathcal G}\!\left(
        \frac{t}{d^{p\theta+1}}
        \,\nabla^{\perp}_{2i-1,2i}d
        \right)
        =
        -\frac{\lambda_i}{2}\,
        \left\langle\frac{\mathbf z^{(i)}}{d^{p\theta+1}},
         \nabla_{2i-1,2i}d \right\rangle
        +
        \lambda_i\,\frac{t}{d^{p\theta+1}}\,\partial_t d .
    \end{equation}
Dividing by $\lambda_i$ and summing over $i$, we obtain \textnormal{(i)} on $\mathcal G\setminus L$.

To complete the proof, fix $\varepsilon>0$ and set \(D_\varepsilon=\{(z,t)\in\mathcal G : |z|>\varepsilon\}.\)
For \(\phi\in C_c^\infty(\mathcal G\setminus\{0\})\),
\begin{multline*}
    \int_{D_\varepsilon} \Div_{\mathcal G}\left( \frac{t}{d^{p\theta+1}} B^{-1}\nabla_{\mathcal G} d \right) \phi \, dz\,dt\\
    =
    -\!\int_{D_\varepsilon}
    \frac{t}{d^{p\theta+1}}\, \langle
    B^{-1}\nabla_{\mathcal G}d , \nabla_{\mathcal G}\phi\rangle\, dz\,dt
    +
    \int_{|z|=\varepsilon}
    \frac{t}{d^{p\theta+1}}\,\phi\, \left\langle
    B^{-1}\nabla_{\mathcal G}d,\frac{z}{|z|} \right\rangle
    \, dH_{2n}.
\end{multline*}
Observe that the boundary integral on the r.h.s. vanishes as \(\varepsilon\to0^+\) by assumption \emph{(d.3)}, since $\phi$ is compactly supported away from zero.
Hence, since $D_\varepsilon\subset \mathcal G\setminus L$, applying identity \textnormal{(i)} on the l.h.s.~and letting \(\varepsilon\to0^+\), we obtain \textnormal{(i)} in the sense of distributions on \(\mathcal G\setminus\{0\}\).

The identity in \textnormal{(ii)} is proved in a similar way.
\end{proof}






\section{Proof of Theorem~\ref{thm:-Hardy-lungo-Z_d}}
\label{sec:proof-theorem}

In this section, we prove Theorem~\ref{thm:-Hardy-lungo-Z_d}, whose statement we recall for the reader's convenience using the notations \eqref{eq:Definizione-nabla-2i-1-2i}.

\begin{theorem*}
Let \( p \geq 2 \), \( \theta \in \mathbb{R} \), and let \( d \) be a regular positive homogeneous function on \( \mathcal{G} \). Then, for every \( u \in C_c^\infty(\mathcal{G} \setminus \{0\}) \), the following Hardy-type inequality holds
\begin{equation}
\int_{\mathcal{G}} \frac{|\left\langle \nabla_{\mathcal{G}} u , Z_d \right\rangle|^p}{d^{p(\theta - 1)}} \, dz\,dt 
\geq \left| \frac{Q - p\theta}{p} \right|^p \int_{\mathcal{G}} \frac{|u|^p}{d^{p\theta}} \, dz\,dt.
\end{equation}
Here, the vector field $Z_d$ is defined by
\[
Z_d
=
\frac{n+1}{n}\,\frac{z}{d}
-
\frac{2p\theta}{n}\,\frac{t}{d^2}
\sum_{i=1}^n \frac{1}{\lambda_i}\,\nabla_{2i-1,2i}^{\perp} d.
\]

%
If, additionally, \( d \) is such that $\langle z,B^{-1}\nabla_z d\rangle=0$ for all $z\in \mathcal G\setminus\{0\}$, 
then inequality \eqref{eq:-HardyZ_d} is sharp, and equality is attained by 
\[
u(z, t) = \left( \frac{|t|}{|z|^2} \right)^{\frac{Q - 2}{2p}}.
\]
\end{theorem*}

We start with the following, which links the horizontal vector field \(Z_d\) with the non-horizontal Euler vector field \(\mathcal{E}\).

\begin{proposition}\label{prop:-identità-tra-Z-e-E}
Let \(p \geq 2\), \(\theta \in \mathbb{R}\), and let \( d \) be a regular positive homogeneous function on \( \mathcal{G} \). Then, for every \(u \in C_c^\infty(\mathcal{G} \setminus \{0\})\),
\begin{equation}\label{eq:-equazione-identità-Lp-tra-Z-e-E}
\int_{\mathcal{G}} \frac{|u|^{p-2} u}{d^{p\theta - 1}}\, 
\langle\nabla_{\mathcal{G}} u , Z_d \rangle\, dz\, dt
=
\int_{\mathcal{G}} \frac{|u|^{p-2} u\, \mathcal{E} u}{d^{p\theta}} \, dz\, dt
=
-\frac{Q - p\theta}{p} 
\int_{\mathcal{G}} \frac{|u|^p}{d^{p\theta}} \, dz\, dt.
\end{equation}
\end{proposition}

\begin{proof}
By the definition of \(Z_d\), on \(\mathcal G\setminus\{0\}\) we have
\[
\Div_{\mathcal G}\!\left(\frac{1}{d^{p\theta-1}}\,Z_d\right)
=
\frac{n+1}{n}\,
\Div_{\mathcal G}\!\left(\frac{z}{d^{p\theta}}\right)
-
\frac{2p\theta}{n}
\sum_{i=1}^n
\frac{1}{\lambda_i}\,
\Div_{\mathcal G}\!\left(
\frac{t}{d^{p\theta+1}}\,
\nabla^{\perp}_{2i-1,2i} d
\right).
\]
Substituting identities \textnormal{(i)}–\textnormal{(ii)} of Lemma~\ref{lemma:div-dist}, and using
\(\mathcal E d = d\), gives
\[
\Div_{\mathcal G}\!\left(\frac{1}{d^{p\theta-1}}\,Z_d\right)
=
\frac{2n+2}{d^{p\theta}}
-
\frac{p\theta}{d^{p\theta}}\, 
=
\frac{Q-p\theta}{d^{p\theta}} .
\]
Hence,
\[
\int_{\mathcal G}
\frac{|u|^{p-2}u}{d^{p\theta-1}}\,
\langle \nabla_{\mathcal G}u, Z_d \rangle\, dz\,dt
=
\frac1p
\int_{\mathcal G}
\left\langle\nabla_{\mathcal G}|u|^p,
\frac{Z_d}{d^{p\theta-1}}\right\rangle\, dz\,dt
=
-\frac{Q-p\theta}{p}
\int_{\mathcal G}
\frac{|u|^p}{d^{p\theta}}\, dz\,dt .
\]

Finally, from the adjoint relation \(\mathcal E^{*}=-Q-\mathcal E\), we have
\[
\int_{\mathcal G}
\frac{|u|^{p-2}u\,\mathcal E u}{d^{p\theta}}\, dz\,dt
=
-\frac{Q-p\theta}{p}
\int_{\mathcal G}
\frac{|u|^p}{d^{p\theta}}\, dz\,dt.
\]
\end{proof}

For the proof of Theorem~\ref{thm:-Hardy-lungo-Z_d} we combine an algebraic
$L^p$-identity, a suitable change of variables, and a family of radial
cut-off functions.
We recall first the following algebraic identity.
\begin{proposition}[Proposition 2.1 in \cite{DA}]
    \label{prop:-identità-fondamentale}
    For every domain \( \Omega \subseteq \mathbb{R}^N \) and every \( f,g \in L^p(\Omega) \) with \( p \ge 2 \),
    \begin{equation}
        \label{eq:-identità-fondamentale}
        \| w(p,f,g)(f-g) \|_{L^2(\Omega)}^2
        =
        \|f\|_{L^p(\Omega)}^p
        + (p-1)\|g\|_{L^p(\Omega)}^p
        - p (|g|^{p-2}g, f)_{L^2(\Omega)} ,
    \end{equation}
    where
    \[
        w(p,f,g)^2
        =
        p(p-1)\int_0^1 s\,|s g + (1-s)f|^{p-2}\,ds.
    \]
\end{proposition}

\begin{remark}
\label{rmk:-norma-L2-non-lineare-uguale-a-0}
Observe that
\(
\| w(p,f,g)(f-g) \|_{L^2} = 0
\)
if and only if \(f=g\) almost everywhere in \( \Omega \).
\end{remark}

We will exploit the change of variables $\Phi:\mathbb{R}^{2n}\setminus\{0\} \times \mathbb{R} \to \mathcal G \setminus L$, which has been introduced in the Heisenberg group setting in \cite{FMRS}, and is given by
\begin{equation}
    \label{eq:-change-of-variables-Phi}
\Phi(\omega,\lambda)=\left(
\frac{\omega}{(1+\lambda^2)^{1/4}},
\frac{\lambda}{(1+\lambda^2)^{1/2}}\,|\omega|^2
\right),
\end{equation}
Its inverse is
\[
(\omega,\lambda)
=
\Phi^{-1}(z, t) =
\left(
\rho\,\frac{z}{|z|},
\frac{t}{|z|^2}
\right),
\qquad
\rho = (|z|^4 + t^2)^{1/4},
\]
so that \( |\omega| = \rho \).
Computing the Jacobian determinant of $\phi$ yields that, for every integrable \(f\),
\begin{equation}
\label{eq:-change-of-variables-integrale}
\int_{\mathcal G} f(z,t)\,dz\,dt
=
\int_{\mathbb{R}^{2n}\setminus\{0\}\times\mathbb{R}}
f(\Phi(\omega,\lambda))
\frac{|\omega|^2}{(1+\lambda^2)^{\frac{n+1}{2}}}
\, d\omega\,d\lambda .
\end{equation}

We are finally in a position to complete the proof of Theorem~\ref{thm:-Hardy-lungo-Z_d}.

\begin{proof}[Proof of Theorem~\ref{thm:-Hardy-lungo-Z_d}]
Equality \eqref{eq:-HardyZ_d} follows by direct application of Proposition~\ref{prop:-identità-fondamentale} 
with
\[
    f = \frac{\langle \nabla_{\mathcal G} u , Z_d \rangle}{d^{\theta - 1}},
    \qquad
    g = -\frac{Q - p\theta}{p}\,\frac{u}{d^{\theta}}.
\]
Indeed, using Proposition~\ref{prop:-identità-tra-Z-e-E} to evaluate the mixed term, yields
\[
    \int_{\mathcal G} \frac{|\langle \nabla_{\mathcal G} u , Z_d\rangle|^p}{d^{p(\theta - 1)}}\,dzdt
    - \left|\frac{Q - p\theta}{p}\right|^p
      \int_{\mathcal G} \frac{|u|^p}{d^{p\theta}}\,dzdt
    = \| w(p,f,g)(f-g) \|_{L^2(\mathcal G)}^2
    \ge 0.
\]

In the case where $\langle z , B^{-1}\nabla_z d\rangle = 0$, we prove the sharpness of \eqref{eq:-HardyZ_d}, by characterizing the extremals.
To this aim, we observe that equality in \eqref{eq:-HardyZ_d} requires \(f=g\),
which amounts to
\begin{equation}
\label{eq:-funz_massimizzanti}
\left\langle\nabla_{\mathcal G} u ,
\frac{n+1}{n}\frac{z}{d^\theta}
-
\frac{2p\theta}{n}\frac{t}{d^{\theta+1}}\,B^{-1}\nabla_{\mathcal G} d
 \right\rangle
+
\frac{Q-p\theta}{p}\frac{u}{d^\theta}
=0 .
\end{equation}

Let \(u(z,t)=\varphi(t/|z|^2)\) and compute 
\[
\nabla_{\mathcal G}u
=
\varphi'\!\left(\frac{t}{|z|^2}\right)
\left(
-2\,\frac{t}{|z|^4}z
+
\frac{1}{2|z|^2}Bz
\right).
\]
Then, Lemma~\ref{lemma:-ricostruzione_d} yields
\[
\langle \nabla_{\mathcal G}u, Z_d\rangle
=
\frac{2(p\theta-Q)}{Q-2}\,
d^{-\theta}\,
\frac{t}{|z|^2}\,
\varphi'\!\left(\frac{t}{|z|^2}\right).
\]
Equation \eqref{eq:-funz_massimizzanti} therefore reduces to
\[
\frac{t}{|z|^2}\,
\varphi'\!\left(\frac{t}{|z|^2}\right)
=
\frac{Q-2}{2p}\,
\varphi\!\left(\frac{t}{|z|^2}\right).
\]
This shows that the extremal profile is
\(u(z,t)
=
\left(\frac{|t|}{|z|^{2}}\right)^{\frac{Q-2}{2p}}\),
up to a multiplicative constant.

It remains to implement a cut-off argument.  
To this aim, we fix a family of radial cut-off functions \( g_\varepsilon \in C_c^\infty(\mathbb{R}^+) \), \(\varepsilon>0\), 
such that
\[
0 \le g_\varepsilon(r) \le 1,\qquad
g_\varepsilon(r)=
\begin{cases}
0 & 0 \le r \le \varepsilon \ \text{or}\ r \ge 1/\varepsilon,\\
1 & 2\varepsilon \le r \le 1/(2\varepsilon),
\end{cases}
\]
and
\[
|g_\varepsilon'(r)| \le 
\begin{cases}
c/\varepsilon & \varepsilon \le r \le 2\varepsilon,\\[0.3em]
c\,\varepsilon & 1/(2\varepsilon) \le r \le 1/\varepsilon,
\end{cases}
\]
for some constant \(c>0\) independent of \( \varepsilon \).
Fix \(\eta\in C^\infty_c(\mathbb R\setminus\{0\})\) and set
\[
u_\varepsilon(z,t)
=
\left(\frac{|t|}{|z|^2}\right)^{\frac{Q-2}{2p}}
g_\varepsilon\!\left(\frac{t}{|z|^2}\right)\eta(d(z,t)),
\qquad \varepsilon>0,
\]
so that \(u_\varepsilon\in C^\infty_c(\mathcal G\setminus L)\).

Since positive homogeneous functions on \(\mathcal G\) are equivalent (Section 5.1, \cite{BLU}), 
there exists \(C>0\) such that, with \(\rho=(|z|^4+t^2)^{1/4}\),
\begin{equation}
\label{eq:-equiv_norme}
\frac{1}{C}\,\rho \le d \le C\,\rho .
\end{equation}
We can thus choose $\eta$ such that there exists radii \(0<r_1<R_1\) and \(0<r_2<R_2\) such that
\[
\{r_1\le \rho \le R_1\}\subseteq \{\eta(d)\ge 1\},
\qquad
\operatorname{supp}\eta(d)\subseteq \{r_2\le \rho \le R_2\}.
\]

In what follows, \( \mathrm{c}>0 \) denotes a constant that may vary from line to line.
Using the change of variables introduced in \eqref{eq:-change-of-variables-Phi}, and the corresponding formula \eqref{eq:-change-of-variables-integrale}, we compute
\[
\int_{\mathcal G} \frac{|u_\varepsilon|^p}{d^{p\theta}}\, dz \, dt
=
\int_{r_1\le |\omega|\le R_1}
\int_{0}^{+\infty}
\frac{|\omega|^2}{d^{p\theta}(\Phi(\omega,\lambda))} 
\frac{\lambda^{\frac{Q-2}{2}}}{(1+\lambda^2)^{\frac{n+1}{2}}}
|g_\varepsilon(\lambda)|^p \, d\lambda\, d\omega .
\]
Considering polar coordinates for $\omega\in \mathbb{R}^{2n}\setminus \{0\}$, and recalling that $|\omega|=\rho$, we have that
\[
    \int_{r_1\le |\omega|\le R_1} \frac{|\omega|^2}{d^{p\theta}(\Phi(\omega,\lambda))} \, d\omega
    \ge
    \frac{|\mathbb{S}^{2n-1}|}{C^{p\theta}}\int_{r_1\le \rho \le R_1} {\rho^{2-p\theta+2n-1}} \, d\rho
    \ge \mathrm{c}>0.
\]
On the other hand, for \(\lambda\) sufficiently large, we have that $\lambda^{Q-2}/(1+\lambda^2)^{\frac{n+1}{2}}\ge \mathrm{c}\lambda^{-1}$. 
Hence,
\[
\int_{\mathcal G}\frac{|u_\varepsilon|^p}{d^{p\theta}}
\ge
\mathrm{c}
\int_{\lambda_0}^{\frac{1}{2\varepsilon}} \frac{1}{\lambda}\,d\lambda 
= \mathrm{c}\, \log\!\left(\frac{1}{2\lambda_0\varepsilon}\right).
\]

We are left to compute the l.h.s. of \eqref{eq:-HardyZ_d} with \(u_\varepsilon\).
We have,
\[
\left|\frac{\langle \nabla_{\mathcal G}u_\varepsilon, Z_d\rangle}{d^\theta}\right|^p
= |a+b+c|^p,
\]
with
\begin{eqnarray*}
    a &=& -\frac{Q - p \theta}{p} \left( \frac{|t|}{|z|^2} \right)^{\frac{Q-2}{2p}} g_\varepsilon\!\left( \frac{t}{|z|^2} \right) \eta(d) \frac{1}{d^\theta}, \\
    b &=& -\frac{2 (Q - p \theta)}{Q - 2} \left( \frac{|t|}{|z|^2} \right)^{\frac{Q-2}{2p}} \frac{t}{|z|^2} \frac{1}{d^\theta} g_\varepsilon'\!\left( \frac{t}{|z|^2} \right) \eta(d), \\
    c &=& \frac{Q}{Q - 2} \frac{\langle z , \nabla_{\mathcal{G}} d\rangle}{d^\theta} \left( \frac{|t|}{|z|^2} \right)^{\frac{Q-2}{2p}} g_\varepsilon\!\left( \frac{t}{|z|^2} \right) \eta'(d).
\end{eqnarray*}
Invoking Corollary~2.1 in~\cite{DA}, we use that  
for all \(a,b,c\in\mathbb{R}\) there exists a constant \(c_p>0\) such that
\begin{equation*}
|a+b+c|^p
\le |a|^p
+ c_p |a|^{p-1}(|b|+|c|)
+ c_p |b|^p
+ c_p^2 |b|^{p-1}|c|
+ c_p^2 |c|^p .
\end{equation*}
The contribution of \(|a|^p\) is explicit:
\begin{equation*}
\int_{\mathcal{G}} |a|^p \,dz\,dt
= \left|\frac{Q-p\theta}{p}\right|^p
\int_{\mathcal{G}} \frac{|u_\varepsilon|^p}{d^{p\theta}}\,dz\,dt .
\end{equation*}
The remaining terms are estimated using the structure of \(g_\varepsilon\) 
and the support of \(\eta\).  
After the change of variables \(\lambda = t/|z|^2\), each integral involving 
\(b\) or \(c\) is bounded by an expression of the form
\begin{equation}\label{eq:model-integral}
\int_0^\infty 
\frac{\lambda^\alpha}{(1+\lambda^2)^{Q/4}}\,
|g_\varepsilon(\lambda)|^\beta\,
|g'_\varepsilon(\lambda)|^\gamma \, d\lambda ,
\end{equation}
where \(\beta,\gamma\ge 0\) and
\begin{equation}\label{eq:alpha-gamma-cases}
\alpha = \frac{Q-2}{2} + \gamma \quad \text{if } \gamma>0,
\qquad
\alpha = \frac{Q-3}{2} \quad \text{if } \gamma=0.
\end{equation}

To justify these exponents we note that on \(\operatorname{supp}\eta(d)\) one has 
\(|z|,|t|\le R_2\) and \(|\nabla_{\mathcal{G}} d|\le \mathrm{c}\). Hence
\[
|\langle z,\nabla_{\mathcal{G}} d\rangle|
\le \mathrm{c}\,\frac{|z|}{\sqrt{|t|}}
= \mathrm{c}\,\lambda^{-1/2}
\]
in the term involving \(|a|^{p-1}|c|\), and
\[
|\langle z,\nabla_{\mathcal{G}} d\rangle|^p
\le \frac{|z|}{\sqrt{|t|}}\,|z|^{p-1}\sqrt{|t|}\,|\nabla_{\mathcal{G}} d|^p
\le  \mathrm{c}\,\lambda^{-1/2}
\]
in the term involving \(|c|^p\).
In contrast, for the term \(|b|^{p-1}|c|\), the powers produced by 
\((|t|/|z|^2)^{\cdot}\) give the required exponent \(\alpha\), and the uniform bound
\(
|\langle z,\nabla_{\mathcal{G}} d\rangle|\le \mathrm{c}
\)
is sufficient.

A direct inspection of \eqref{eq:model-integral} using 
\eqref{eq:alpha-gamma-cases} shows that all such terms are uniformly bounded as 
\(\varepsilon\to 0\). Hence
\begin{equation*}
\int_{\mathcal{G}} |a|^{p-1}|b|\,dz\,dt = \mathcal{O}(1), \qquad
\int_{\mathcal{G}} |a|^{p-1}|c| \,dz\,dt= \mathcal{O}(1), \qquad
\int_{\mathcal{G}} |b|^p\,dz\,dt = \mathcal{O}(1),
\end{equation*}
and similarly
\begin{equation*}
\int_{\mathcal{G}} |b|^{p-1}|c|\,dz\,dt = \mathcal{O}(1), \qquad
\int_{\mathcal{G}} |c|^p\,dz\,dt =\mathcal{O}(1).
\end{equation*}
Collecting the preceding estimates gives
\begin{multline*}
    \left|\frac{Q-p\theta}{p}\right|^{p}
    \le 
    \inf_{u\in C_c^\infty(\mathcal{G}\setminus\{0\})}
    \frac{\displaystyle\int_{\mathcal{G}}
    \frac{|\langle \nabla_{\mathcal{G}}u, Z_d\rangle|^{p}}{d^{p(\theta-1)}}\,dz\,dt}
    {\displaystyle\int_{\mathcal{G}}
    \frac{|u|^{p}}{d^{p\theta}}\,dz\,dt}\\
    \le 
    \frac{\displaystyle\int_{\mathcal{G}}
    \frac{|\langle\nabla_{\mathcal{G}}u_\varepsilon, Z_d\rangle|^{p}}{d^{p(\theta-1)}}\,dz\,dt}
    {\displaystyle\int_{\mathcal{G}}
    \frac{|u_\varepsilon|^{p}}{d^{p\theta}}\,dz\,dt}
    =
    \left|\frac{Q-p\theta}{p}\right|^{p}
    + \frac{\mathcal{O}(1)}{-\log(2 \lambda_0 \varepsilon)}.
\end{multline*}
Letting \(\varepsilon\to 0\) concludes the proof of the sharpness of \eqref{eq:-HardyZ_d} and thus of the statement.
\end{proof}

\section{Unweighted Hardy inequality}
\label{sec:unweighted-hardy-inequality}

In this section we obtain explicit lower bounds for the quantity defined in \eqref{eq:cdp}, i.e., the optimal constant $c(d,p,\theta)$ for the Hardy inequality  
\begin{equation}\label{eq:unweighted_hardy_general}
\int_{\mathcal{G}} \frac{|\nabla_{\mathcal{G}}u|^{p}}{d^{p(\theta-1)}}\,dz\,dt
\;\ge\;
c \int_{\mathcal{G}} \frac{|u|^{p}}{d^{p\theta}}\,dz\,dt,
\qquad
u\in C_c^\infty(\mathcal{G}\setminus\{0\}),\ \theta\in\mathbb{R}.
\end{equation}
This is achieved via Corollary~\ref{corr:lower-bound}, by providing upper bounds for the vector field \(Z_d\) defined in \eqref{eq:-definizione-di-Z_d} for specific choices of the homogeneous function \(d\).

It is convenient to distinguish the isotropic regime $\lambda_1=\dots=\lambda_n$ 
from the genuinely non-isotropic one.
In the former case the matrix $B$ takes the block diagonal form
\begin{equation}\label{eq:-B-caso-isotropo}
B=\lambda_1
\begin{pmatrix}
0 & 1\\
-1 & 0
\end{pmatrix}
\oplus \cdots \oplus
\begin{pmatrix}
0 & 1\\
-1 & 0
\end{pmatrix},
\end{equation}
and the change of variables
\[
(z,t)\mapsto (\eta,\tau)=\bigl(z,\tfrac{\lambda_1}{4}\,t\bigr)
\]
reduces the structure to that of the standard Heisenberg group.
The analysis in the isotropic setting will be carried out in the next section, 
while the non-isotropic case will be treated separately.

\subsection{The Heisenberg group}\label{sec:heisenberg}
For \(n\in\mathbb{N}\), the Heisenberg group is 
\(\mathbb{H}^n=(\mathbb{R}^{2n+1},\circ)\), 
with group law associated with the matrix \(B\) in~\eqref{eq:-B-caso-isotropo}, 
where we fix \(\lambda_1=4\).
The Lie algebra is generated by the left-invariant vector fields
\[
X_{2i-1}=\partial_{z_{2i-1}}+2z_{2i}\,\partial_t,\qquad
X_{2i}=\partial_{z_{2i}}-2z_{2i-1}\,\partial_t,\qquad
\partial_t,
\quad i=1,\dots,n,
\]
with the only non-trivial commutation relations
\[
[X_{2i-1},X_{2i}]=-4\,\partial_t.
\]
We denote the horizontal gradient by
\[
\nabla_{\mathbb{H}}=(X_1,\dots,X_{2n}).
\]
In this setting, the matrix $B$ satisfies
\[
B^{-1}=4J,
\]
where $J$ denotes the canonical skew-symmetric matrix acting as a rotation by $\pi/2$ on each horizontal two-dimensional subspace, that is, $J^2=-I$.
Accordingly, we set
\[
\nabla_{\mathbb H}^{\perp}u := J\,\nabla_{\mathbb H}u.
\]
The non-isotropic dilations
\[
\delta_\gamma(z,t)=(\gamma z,\gamma^2 t),\qquad \gamma>0,
\]
are group automorphisms and determine the homogeneous structure of 
\(\mathbb{H}^n\), whose homogeneous dimension is \(Q=2n+2\).

Two explicit regular positive homogeneous functions will be considered on \(\mathbb{H}^n\): the Korányi gauge $\rho = (|z|^4 + t^2)^{1/4}$ and the Carnot--Carathéodory distance $\delta_{cc}$. The former is associated with the fundamental solution of the Heisenberg sub-Laplacian \cite{follandFundamental1973}, while the latter is the natural distance induced by the horizontal vector fields \cite{agrachevComprehensive2019}. 

\begin{proof}
    [Proof of Theorem~\ref{thm:constant-hardy-heisnberg}]

    After showing that $\rho$ and $\delta_{cc}$ satisfy the assumptions of Theorem~\ref{thm:-Hardy-lungo-Z_d}, we will provide explicit upper bounds for the vector field \(Z_d\) associated with each of them, thus leading to the desired lower bounds for \(c(d,p,\theta)\) via Corollary~\ref{corr:lower-bound}.
    
    Let us start by considering the Koranyi gauge \(\rho\). 
The vector field defined in \eqref{eq:-definizione-di-Z_d} takes the form
\[
Z_{\rho}
=\frac{Q}{Q-2}\,\frac{z}{\rho}
-\frac{p\theta}{Q-2}\,\frac{t}{\rho^{2}}\,
\nabla_{\mathbb H}^{\perp}\rho.
\]
For the Korányi gauge, the identities
\[
|\nabla_{\mathbb H}\rho|^{2}=\frac{|z|^{2}}{\rho^{2}},
\qquad
\langle z,\nabla_{\mathbb H}^{\perp}\rho\rangle=\frac{|z|^{2}t}{\rho^{3}}
\]
follow from direct computations.
Together with Remark~\ref{rmk:hom-norm}, they imply that assumptions \textup{(d.1)--(d.3)} are satisfied.\\

Introducing the parameter $\lambda=t/|z|^{2}$, a straightforward computation yields
\[
|Z_\rho|^{2}
=\frac{1}{(1+\lambda^{2})^{1/2}}
\left[
\left(\frac{Q}{Q-2}\right)^{2}
+\frac{p\theta(p\theta-2Q)}{(Q-2)^{2}}
\frac{\lambda^{2}}{1+\lambda^{2}}
\right].
\]
Maximization over \(\lambda\in\mathbb{R}\) leads to the desired bound \eqref{eq:cdp-koranyi}.


We now turn to the Carnot--Carathéodory distance \(\delta_{cc}\). Recall that $\delta_{cc}$ is continous and smooth outside the center $\{z=0\}$, see, e.g., \cite[Theorem~3.1]{HZ}.
For later use, it is convenient to introduce polar coordinates adapted to \(\delta_{cc}\),
we follow~\cite{AR} (see also~\cite{FP}).
Let $S = \{ w \in \mathbb{C}^n : |w|=1 \}$ be the unit sphere in \(\mathbb{C}^n\), and define 
\[
\Phi:\mathcal S\times[-2\pi,2\pi]\times(0,+\infty)\to \mathbb{H}^n,
\qquad 
(a+ib,\nu,r)\mapsto (z,t),
\]
where $a,b\in \mathbb{R}^n$ and, for \(\nu\neq0\), we let
\[
\begin{cases}
\displaystyle 
z_{2i-1}
= \frac{b_i(1-\cos\nu)+a_i\sin\nu}{\nu}\, r,\\[10pt]
\displaystyle 
z_{2i}
= \frac{-a_i(1-\cos\nu)+b_i\sin\nu}{\nu}\, r,\\[10pt]
\displaystyle 
t = 2\,\frac{\nu-\sin\nu}{\nu^2}\, r^2,
\end{cases}
\]
while, for \(\nu=0\), we let
\[
z_{2i-1} = a_i \, r, \quad z_{2i} = b_i \, r, \quad t = 0,
\]
Then,
\[
\delta_{cc}(\Phi(a+ib,\nu,r)) = r.
\]

We record the following identities for the derivatives of \(\delta_{cc}\)
(Lemma~3.11 in~\cite{AR}).
If \((z,t)=\Phi(a+ib,\nu,r)\) with \(\nu\neq0\), then
\begin{equation}\label{eq:X-der-deltacc}
X_{2i-1}\delta_{cc}=b_i\sin\nu+a_i\cos\nu,\qquad
X_{2i}\delta_{cc}=b_i\cos\nu-a_i\sin\nu,
\end{equation}
\begin{equation}\label{eq:t-der-deltacc}
\partial_t\delta_{cc}=\frac{\nu}{4r}.
\end{equation}
We also use the relations
\begin{equation}\label{eq:z2-and-perp}
|z|^{2}=2\,\frac{1-\cos\nu}{\nu^{2}}\,r^{2},
\qquad
z\cdot\nabla_{\mathbb H}^{\perp}\delta_{cc}
=\frac{1-\cos\nu}{\nu}\, r,
\end{equation}
whose verification follows directly from the parametrisation \(\Phi\).
In particular,
\[
|\nabla_{\mathbb H}\delta_{cc}|=1,
\qquad
|\partial_t\delta_{cc}|\in L^\infty_{\mathrm{loc}}
(\mathbb H^n\setminus\{0\}).
\]

Combining \eqref{eq:t-der-deltacc}–\eqref{eq:z2-and-perp} yields
\[
\left\langle\frac{z}{|z|},\nabla_{\mathbb H}^{\perp}\delta_{cc}\right\rangle
=2|z|\,\partial_t\delta_{cc}.
\]
Since the right-hand side tends to zero as \(|z|\to0\),
this implies that assumption \emph{(d.3)} is satisfied.
%
Moreover,
\[
\langle z, B^{-1}\nabla_{\mathbb H}\delta_{cc}\rangle
=\langle \nabla_{\mathbb H}^{\perp}\delta_{cc}, z\rangle
-2|z|^{2}\,\partial_t\delta_{cc}
=0,
\]
so that \(\delta_{cc}\) satisfies all structural assumptions required in
Theorem~\ref{thm:-Hardy-lungo-Z_d}, and we can apply Corollary~\ref{corr:lower-bound}.

In this case, we have that
\[
Z_{\delta_{cc}}
=\frac{Q}{Q-2}\frac{z}{\delta_{cc}}
-\frac{p\theta}{Q-2}\frac{t}{\delta_{cc}^{2}}\,
\nabla_{\mathbb H}^{\perp}\delta_{cc}.
\]
Hence, the polar representation \((z,t)=\Phi(a+ib,\nu,r)\) together with
\eqref{eq:z2-and-perp}–\eqref{eq:t-der-deltacc}
yields
\[
|Z_{\delta_{cc}}(z,t)|^{2}
=
2\!\left(\frac{Q}{Q-2}\right)^{2}\frac{1-\cos\nu}{\nu^{2}}
+
\left(\frac{2p\theta}{Q-2}\right)^{2}
\!\left(\frac{\nu-\sin\nu}{\nu^{2}}\right)^{2}
-
\frac{4p\theta\,Q}{(Q-2)^{2}}
\frac{\nu-\sin\nu}{\nu^{2}}\,
\frac{1-\cos\nu}{\nu}.
\]
Since the above is independent of $a,b$ and $r$, we set
\[
g(\nu)=|Z_{\delta_{cc}}(\Phi(a+ib,\nu,r))|^{2},
\qquad \nu\in[-2\pi,2\pi].
\]
The function \(g\) is continuous and positive on a compact interval; hence it admits a strictly positive and finite supremum.
In particular, this implies that
\begin{equation}
    c(\delta_{cc},p,\theta) \ge \frac{1}{\|g\|_\infty^{p/2}} \left| \frac{Q - p \theta}{p} \right|,
\end{equation}
proving the second part of \eqref{eq:cdp-carnot-caratheodory}.
A closed expression for this supremum does not seem accessible in general,
due to the nonlinear dependence on \((p,\theta,Q)\).

To prove the first bound in \eqref{eq:cdp-carnot-caratheodory},
we show that
the maximum of \(g\) on \([-2\pi,2\pi]\) is attained at \(\nu = 0\) if we assume that 
\[
Q \ge \frac{4 p \theta}{12 - \pi^2} \ge 0.
\]
Set
\[
 \mathfrak f(\nu) := \frac{1 - \cos \nu}{\nu^2}, 
\qquad 
\mathfrak g(\nu) := \frac{\nu - \sin \nu}{\nu^2}, 
\qquad 
\mathfrak h(\nu) := p \theta \, \mathfrak g(\nu) - Q \nu \, \mathfrak f(\nu).
\]
Then,
\[
g(\nu)
=
2\left(\frac{Q}{Q - 2}\right)^2 \mathfrak f(\nu) 
+ \frac{4 p \theta}{(Q - 2)^2} \, \mathfrak g(\nu) \mathfrak h(\nu).
\]
Since both \(\mathfrak g\) and \(\mathfrak h\) are odd, and for 
\(0<\nu<2\sqrt{3 - {p\theta}/{Q}}\) one has
\[
\mathfrak f(\nu)\ge \frac12 - \frac{\nu^{2}}{24}>0, 
\qquad 
\mathfrak g(\nu)\le \frac{\nu}{6}, 
\qquad 
\mathfrak h(\nu) 
\leq \nu\left( \frac{p \theta}{6} - \frac{Q}{2} \right) 
+ \frac{Q}{24} \nu^3 
< 0,
\]
it follows that
\(
\mathfrak g(\nu)\mathfrak h(\nu)\le 0 .
\)
Hence
\[
g(\nu)
\le 
\left(\frac{Q}{Q-2}\right)^{2}
\qquad \text{for } \,|\nu|<2\sqrt{3 - \frac{p\theta}{Q}},
\]
and \(\nu=0\) is a local maximum of
\(g\).

To conclude that \(\nu=0\) yields the global maximum, 
it remains to control the region \(|\nu|\ge\pi\), since the assumption on \(Q\) guarantees
\(
2\sqrt{3 - {p\theta}/{Q}} \ge \pi .
\)
Since \(|Z_{\delta_{cc}}|\) is even in \(\nu\), it suffices to restrict to 
\(\nu\in[\pi,2\pi]\).
On this interval,
\[
f(\nu)\le f(\pi),\qquad 
g(\nu)\le g(\pi)=\frac{1}{\pi},
\]
and therefore
\[
g(\nu)
\le
2\left(\frac{Q}{Q-2}\right)^{2} f(\pi)
+
\frac{4 p^{2}\theta^{2}}{(Q-2)^{2}}\, g(\pi)^{2}
=
\left(\frac{Q}{Q-2}\right)^{2}
\frac{4}{\pi^{2}}
\left(1+\frac{p^{2}\theta^{2}}{Q^{2}}\right).
\]
The right-hand side is bounded above by 
\(\left(\frac{Q}{Q-2}\right)^{2}\).
Consequently,
\begin{equation}\label{eq:max-Z-Carnot}
\sup_{\nu\in[-2\pi,2\pi]} g(\nu)
=
\left(\frac{Q}{Q-2}\right)^{2},
\end{equation}
and the maximum is attained at \(\nu=0\).
This completes the proof.
\end{proof}




\subsection{The non-isotropic case}
\label{sec:non-isotropic-case}
In this section we focus on the case where at least two eigenvalues of $B$ are distinct.
Unlike the isotropic regime, no change of variables can reduce the underlying group law to the Heisenberg structure. This prevents the simplifications exploited in the previous section  and makes the behaviour of the associated vector field \(Z_d\) substantially more delicate. However, for the particular homogeneous norm defined in \eqref{eq:koranyi-generalized}, we can still provide explicit upper bounds for \(|Z_d|\).

\begin{proof}[Proof of Theorem~\ref{thm:constant-hardy-carnot-generalized}]
Since the matrix $B$ is block diagonal as in \eqref{eq:-Definizione-matrice-B-diagonalizzata}, we have
\begin{equation}
    (-B^2)^{1/2} =  \begin{pmatrix} \lambda_1 & 0 & \cdots & 0 & 0  \\
    0 & \lambda_1 & \cdots & 0 & 0 \\
    \vdots & \vdots & \ddots & \vdots &  \vdots \\
    \vdots & \vdots & \cdots & \lambda_n & 0 \\
    0 & 0 & \cdots & 0  & \lambda_n
\end{pmatrix}.
\end{equation}
As a consequence, the norm in \eqref{eq:koranyi-generalized} reads
\[
\rho_B(z,t)=\bigl(|z|_B^{4}+t^{2}\bigr)^{1/4},
\qquad
|z|_B=\sqrt{\sum_{i=1}^{n}\frac{\lambda_i}4(z_{2i-1}^{2}+z_{2i}^{2})}.
\]
It is straightforward to check that \(\rho_B\) is a homogeneous function of degree 1, smooth on \(\mathcal{G}\setminus\{0\}\), and satisfies assumptions \emph{(d.1)}–\emph{(d.3)} in Theorem~\ref{thm:-Hardy-lungo-Z_d}.

The associated vector field takes the form
\begin{equation}
    \label{eq:Z-rho-B}
Z_{\rho_B}(z,t)
=   
\frac{Q}{Q-2}\,\frac{z}{\rho_B}
-\frac{4p\theta}{Q-2}\,\frac{t}{\rho_B^{2}}\,B^{-1}\nabla_{\mathcal{G}}\rho_B.
\end{equation}
In order to bound $c(\rho_B,p,\theta)$ we observe that 
\begin{equation}
    |\langle v, Z_{\rho_B}(z,t)\rangle| 
    \le {|v|}{\sqrt{\frac4{\min \lambda_i}}} |Z_{\rho_B}(z,t)|_B \le 
    |v| \frac2{\sqrt{\min \lambda_i}}\, {\sup_{\mathcal G}|Z_{\rho_B}|_B}.
\end{equation}
Hence, 
\begin{equation}
    c(\rho_B,p,\theta) \ge\left( \frac{\sqrt{\min \lambda_i}}{2\sup_{\mathcal G}|Z_{\rho_B}|_B} \right)^p \left|\frac{Q-p\theta}{p}\right|^p.
\end{equation}
In particular, we reduce the proof to computing the maximum of \(|Z_{\rho_B}(z,t)|_B \).

Direct computations show that 
\begin{gather}
    \label{eq:d-perp-lungo-Bz}
    |B^{-1}\nabla_{\mathcal{G}} \rho_B|_{B}^2 = \sum_{i=1}^n \frac{1}{4\lambda_i} \left( (X_{2i-1} \rho_B)^2 + (X_{2i} \rho_B)^2 \right) = \frac{|z|_B^2}{16\rho_B^2},\\
    \langle z, (-B^2)^{1/2} B^{-1} \nabla_{\mathcal{G}} \rho_B \rangle = \frac{|z|_B^{2}t}{\rho_B^{3}}.
\end{gather}
Observe that, letting $t=\lambda|z|_B^2$, we have $\rho_B(z,\lambda |z|_B^2)=|z|_B(1+\lambda^2)^{1/4}$.
Hence, from \eqref{eq:Z-rho-B},  
\begin{equation}
    g(\lambda) := |Z_{\rho_B}(z,\lambda |z|_B^{2})|_B^{2} = \frac{1}{(1+\lambda^{2})^{1/2}}
    \left[
    \left(\frac{Q}{Q-2}\right)^{2}
    +\frac{p\theta(p\theta-2Q)}{(Q-2)^{2}}
    \frac{\lambda^{2}}{1+\lambda^{2}}
    \right].
\end{equation}
Maximization over \(\lambda\in\mathbb{R}\) leads to the desired bound \eqref{eq:cdp-koranyi-generalized}.
\end{proof}



\begin{remark}
In the isotropic case \(\lambda_i\equiv 4\), one has \(d=2\rho\), so the expression
above reduces to the Heisenberg value obtained with the Korányi gauge.
\end{remark}

The homogeneous norm \(d(z,t)\) considered above is not, in general, 
associated with the fundamental solution of the sub-Laplacian on \(\mathcal{G}\).
In the Heisenberg case, the Korányi norm \(\rho\) arises from the fundamental 
solution and enjoys the property that \(|Z_\rho|\) attains its maximum on the 
horizontal plane \(\{t=0\}\).
It is natural to ask whether this remains true for homogeneous norms generated 
by fundamental solutions on non-isotropic groups.

The next example shows that this is not the case.
Consider \(\mathcal{G}=(\mathbb{R}^4\times\mathbb{R},\circ,\delta_\gamma)\) with 
group law determined by
\[
B=\begin{pmatrix}
0 & \lambda_1 & 0 & 0 \\
-\lambda_1 & 0 & 0 & 0 \\
0 & 0 & 0 & \lambda_2 \\
0 & 0 & -\lambda_2 & 0
\end{pmatrix},
\qquad Q=6.
\]
Beals--Gaveau--Greiner~\cite{BGG} obtained the following integral formula 
for the fundamental solution,
\[
\Gamma(z,t)=
\frac{\lambda_1\lambda_2}{4\pi^3}
\int_{-\infty}^{+\infty}
\frac{s^2\,\mathrm{csch}(2\lambda_1 s)\cosh(2\lambda_2 s)}%
{\bigl(\lambda_1(z_1^2+z_2^2)s\,\coth(2\lambda_1 s)
+ \lambda_2(z_3^2+z_4^2)s\,\coth(2\lambda_2 s)
- i t s\bigr)^2}\,ds.
\]
The corresponding homogeneous norm is \(\rho=\Gamma^{-1/4}\).
For \(\lambda_1=\tfrac12\) and \(\lambda_2=1\), Balogh--Tyson~\cite{BT} gave the 
explicit expression
\[
\rho(z,t)
=
\frac{
\bigl(\bigl(\tfrac{z_1^2+z_2^2}{2}+z_3^2+z_4^2\bigr)^2+t^2\bigr)^{1/8}
\left(
\tfrac{z_1^2+z_2^2}{2}
+\sqrt{
\bigl(\tfrac{z_1^2+z_2^2}{2}+z_3^2+z_4^2\bigr)^2+t^2}
\right)^{3/8}
}{
\left(
\tfrac{z_1^2+z_2^2}{2}+z_3^2+z_4^2
+\sqrt{
\bigl(\tfrac{z_1^2+z_2^2}{2}+z_3^2+z_4^2\bigr)^2+t^2}
\right)^{1/8}
}.
\]
See also~\cite{BZ} for additional details on the explicit derivation.

The behaviour of the associated vector field may differ substantially from the 
isotropic situation.  
In order to illustrate this, consider the case \(p\theta = 2\) and the homogeneous 
norm \(\rho\) described above.  
For this choice, the vector field takes the form
\[
Z_{\rho}
=
\frac{3}{2}\,\frac{z}{\rho}
-
2\,\frac{t}{\rho^{2}}
\left(
-\frac{X_{2}\rho}{\lambda_{1}},\,
 \frac{X_{1}\rho}{\lambda_{1}},\,
-\frac{X_{4}\rho}{\lambda_{2}},\,
 \frac{X_{3}\rho}{\lambda_{2}}
\right),
\]
and therefore
\[
\begin{aligned}
|Z_{\rho}(z,t)|^{2}
&=
\frac{9}{4}\,\frac{|z|^{2}}{\rho^{2}}
+
4\,\frac{t^{2}}{\rho^{4}}
\left(
\frac{(X_{1}\rho)^{2}+(X_{2}\rho)^{2}}{\lambda_{1}^{2}}
+
\frac{(X_{3}\rho)^{2}+(X_{4}\rho)^{2}}{\lambda_{2}^{2}}
\right)
\\
&\quad
-6\,\frac{t}{\rho^{3}}
\left(
\frac{-z_{1}X_{2}\rho + z_{2}X_{1}\rho}{\lambda_{1}}
+
\frac{-z_{3}X_{4}\rho + z_{4}X_{3}\rho}{\lambda_{2}}
\right).
\end{aligned}
\]

Set
\[
B(z,t)=|Z_{\rho}(z,t)|^{2}-|Z_{\rho}(z,0)|^{2}.
\]
Then \(|Z_{\rho}|\) attains its maximum at \(t=0\) if and only if
\(B(z,t)\le 0\).

In contrast with the Heisenberg case, the analytic structure of \(\rho\) 
precludes a direct sign analysis of \(B(z,t)\).
A symbolic computation (performed in \textsc{Mathematica}) shows that
\[
B\bigl((0,0,0,\tfrac14),\,0\bigr)>0.
\]
Thus the maximum of \(|Z_{\rho}|\) is not achieved on the horizontal plane, even 
though \(\rho\) arises from the fundamental solution of the sub-Laplacian.

\subsection{Product of Heisenberg Groups}
\label{sec:product-heisnebrg}
In the Euclidean setting, the sharp Hardy constant is not stable under Cartesian products, for instance, the optimal constant in $\mathbb{R}^4 = \mathbb{R}^2 \times \mathbb{R}^2$ does not coincide with the sum of the two dimensional ones.
This naturally raises the question of what happens in the setting of the Heisenberg group.
Our aim is to understand how the lower bound for the optimal constant in the unweighted Hardy inequality behaves under Cartesian products of Heisenberg groups.

We denote by $\mathbb{H}^n$ the standard Heisenberg group in $\mathbb{R}^{2n+1}$, 
and consider 
\[
\mathcal{G} = (\mathbb{H}^n)^N .
\]
This is a step--two Carnot group with underlying manifold 
$\mathbb{R}^{2nN} \times \mathbb{R}^N$, dilations
\[
\delta_\gamma(z,t) = (\gamma z, \gamma^2 t),
\]
and homogeneous dimension
\[
Q = 2N(n+1).
\]

Writing $z=(z^{(1)},\dots,z^{(N)})$ with $z^{(j)}\in\mathbb{R}^{2n}$ and 
$t=(t_1,\dots,t_N)\in\mathbb{R}^N$, the group law is given componentwise by
\[
(z,t)\circ(\eta,\tau)
=
\Bigl(
z+\eta,\ 
t+\tau
+
\tfrac12\bigl(
\langle B_{\mathbb{H}} z^{(1)},\eta^{(1)}\rangle,\dots,
\langle B_{\mathbb{H}} z^{(N)},\eta^{(N)}\rangle
\bigr)
\Bigr),
\]
where $B_{\mathbb{H}}$ is the standard skew symmetric matrix of the Heisenberg group, given by \eqref{eq:-B-caso-isotropo} with 
\(\lambda_1=4\).

The horizontal layer is generated by
\[
X_{2i-1}^{(j)} 
 = \partial_{z_{2i-1}^{(j)}} + 2 z_{2i}^{(j)}\,\partial_{t_j},
\qquad
X_{2i}^{(j)}
 = \partial_{z_{2i}^{(j)}} - 2 z_{2i-1}^{(j)}\,\partial_{t_j},
\]
for $j=1,\dots,N$ and $i=1,\dots,n$, with the only non-trivial commutation relations
\[
[X_{2i}^{(j)}, X_{2i-1}^{(j)}] = 4\,\partial_{t_j}.
\]
In particular, each vertical vector field can be written as
\[
\partial_{t_j}
=
\frac{1}{4n}\sum_{i=1}^n [X_{2i}^{(j)}, X_{2i-1}^{(j)}].
\]

The infinitesimal generator of dilations on $\mathcal{G}$ is
\[
\mathcal{E}
=
\sum_{j=1}^N
\left(
\sum_{i=1}^n
z_{2i-1}^{(j)}\,\partial_{z_{2i-1}^{(j)}}
+
z_{2i}^{(j)}\,\partial_{z_{2i}^{(j)}}
+
2 t_j\,\partial_{t_j}
\right),
\]
and can be rewritten using the horizontal vector fields as
\begin{equation}\label{eq:ProH-definizione-E}
\mathcal{E}
=
\sum_{j=1}^N
\left(
\sum_{i=1}^n
z_{2i-1}^{(j)}\,X_{2i-1}^{(j)}
+
z_{2i}^{(j)}\,X_{2i}^{(j)}
+
\frac{t_j}{2n}\,[X_{2i}^{(j)}, X_{2i-1}^{(j)}]
\right).
\end{equation}

To state the analogue of Theorem~\ref{thm:-Hardy-lungo-Z_d} in this product setting,
we use the horizontal gradient on the $j$-th factor,
\[
\nabla^{(j)} d=(X_1^{(j)} d,\dots,X_{2n}^{(j)} d),
\qquad
\nabla^{(j)\perp} d=(-X_2^{(j)} d, X_1^{(j)} d,\dots,-X_{2n}^{(j)} d, X_{2n-1}^{(j)} d),
\]
and proceed with the following proposition.

\begin{proposition}
\label{prop:-ProH-Z_d-e-E}
Let $d \in C^\infty(\mathcal{G}\setminus\{0\})$ be a homogeneous norm, and fix
$p\ge 2$, $\theta\in\mathbb{R}$.  
Define
\begin{equation}\label{eq:-ProdH-Z_d}
Z_d
 = \frac{n+1}{n}\,\frac{z}{d}
   - \frac{p\theta}{2n}\,\frac{1}{d^{2}}
     \sum_{j=1}^N t_j\,\nabla^{(j)\perp} d,
\end{equation}
so that the following identity holds for all \( u \in C_c^\infty(\mathcal{G} \setminus \{0\}) \)
\begin{equation}\label{eq: ProH Z_d e E}
\int_{\mathcal{G}} \frac{|u|^{p-2}u\,\mathcal{E}u}{d^{p\theta}} \,dz\,dt
 =
\int_{\mathcal{G}}
\frac{|u|^{p-2}u}{d^{p\theta-1}}\,
\langle\nabla_{\mathcal{G}} u, Z_d \rangle\,dz\,dt
\end{equation}
\end{proposition}
The next result is the corresponding Hardy inequality.

\begin{theorem}
\label{thm:-ProH-Hardy-lungo-Z}
Let $d\in C^\infty(\mathcal{G}\setminus\{0\})$ be a homogeneous norm, and fix $p\ge 2$, $\theta\in\mathbb{R}$.  
Then every $u\in C_c^\infty(\mathcal{G}\setminus\{0\})$ satisfies
\begin{equation}\label{eq:-ProH-Hardy-lungo-Z}
\int_{\mathcal{G}}
\frac{|\langle\nabla_{\mathcal{G}} u ,Z_d\rangle|^p}{d^{p(\theta-1)}}\,dz\,dt
\;\ge\;
\left|\frac{Q-p\theta}{p}\right|^p
\int_{\mathcal{G}}
\frac{|u|^p}{d^{p\theta}}\,dz\,dt ,
\end{equation}
where $Z_d$ is the vector field defined in~\eqref{eq:-ProdH-Z_d}.
\end{theorem}

\begin{proof}
Set
\[
f=\frac{\langle\nabla_{\mathcal{G}}u, Z_d\rangle}{d^{\theta-1}},
\qquad
g=-\frac{Q-p\theta}{p}\,\frac{u}{d^{\theta}} .
\]
By Proposition~\ref{prop:-ProH-Z_d-e-E} and the identity $\mathcal{E}^*=-Q-\mathcal{E}$,
\[
\int_{\mathcal{G}}
\frac{|u|^{p-2}u}{d^{p\theta-1}}\,
\langle \nabla_{\mathcal{G}}u,Z_d \rangle\,dz\,dt
=
\frac{1}{p}
\int_{\mathcal{G}}
\frac{\mathcal{E}(|u|^{p})}{d^{p\theta}}\,dz\,dt
=
-\frac{Q-p\theta}{p}
\int_{\mathcal{G}}
\frac{|u|^{p}}{d^{p\theta}}\,dz\,dt .
\]
Substituting into \eqref{eq:-identità-fondamentale} yields \eqref{eq:-ProH-Hardy-lungo-Z}.
\end{proof}

\begin{proof}[Proof of Proposition~\ref{prop:-ProH-Z_d-e-E}]
An integration by parts in the commutator term gives, for each $j$,
\[
\frac{1}{2n}\sum_{i=1}^n
\int_{\mathcal{G}}
\frac{|u|^{p-2}u\,t_j}{d^{p\theta}}
[X_{2i}^{(j)},X_{2i-1}^{(j)}]u\,dz\,dt
=
\int_{\mathcal{G}}
\frac{|u|^{p-2}u}{d^{p\theta-1}}
\, \left\langle \nabla_{\mathcal{G}}u,
\frac{1}{n}\frac{z^{(j)}}{d}
-\frac{p\theta}{2n}\frac{t_j}{d^2}\nabla^{(j)\perp}d
\right\rangle
dz\,dt .
\]

Using the decomposition
\[
\mathcal{E}
=
\langle z,\nabla_{\mathcal{G}}\rangle
+
\sum_{j=1}^N
\frac{1}{2n}\sum_{i=1}^n
t_j [X_{2i}^{(j)},X_{2i-1}^{(j)}],
\]
we obtain
\[
\int_{\mathcal{G}}
\frac{|u|^{p-2}u\,\mathcal{E}u}{d^{p\theta}}\,dz\,dt
=
\int_{\mathcal{G}}
\frac{|u|^{p-2}u}{d^{p\theta-1}}\, \left\langle
\nabla_{\mathcal{G}}u,\,
\frac{z}{d}
+
\sum_{j=1}^N
\Big[
\frac{1}{n}\frac{z^{(j)}}{d}
-\frac{p\theta}{2n}\frac{t_j}{d^2}\nabla^{(j)\perp}d
\Big]
\right\rangle
dz\,dt .
\]

The vector field inside the parentheses is exactly $Z_d$.
\end{proof}

\medskip

We apply Theorem~\ref{thm:-ProH-Hardy-lungo-Z} to the homogeneous norm
\begin{equation}
\label{eq:-ProH-Definizione-di-p}
\rho(z,t) = \left( |z|^4 + |t|^2 \right)^{1/4}.
\end{equation}
which is the homogeneous gauge naturally induced by the fundamental solution.

For every $j$, a direct computation gives
\[
|\nabla^{(j)}\rho|^2
= \frac{|z^{(j)}|^2}{\rho^6}\big(|z|^4 + t_j^2\big),
\qquad
\langle z,\nabla^{(j)\perp}\rho\rangle
= \frac{t_j |z^{(j)}|^2}{\rho^3}.
\]
Substituting these identities into the expression of $Z_d$ yields
\[
|Z_\rho|^2
=
\left(\frac{n+1}{n}\right)^2\frac{|z|^2}{\rho^2}
+
\frac{p\theta}{n^2}\,\sum_{j=1}^N\,
\frac{t_j^2 |z^{(j)}|^2}{\rho^6}
\,
\left(
\frac{p\theta}{4}\frac{|z|^4+t_j^2}{\rho^4}
-(n+1)
\right).
\]

Since $|z|^4+t_j^2\le \rho^4$, the summation term is nonpositive whenever \(\theta \geq 0\) and
\[
n+1 \ge \frac{p\theta}{4}.
\]
In this regime, $|Z_\rho|$ attains its maximum at $t=0$, and
\[
\sup_{(z,t)} |Z_\rho| = \frac{n+1}{n}.
\]

\begin{corollary}
Let $p\ge2$, $\theta\ge0$, and let \( \rho(z,t) \) be the Korányi norm on \( \mathcal{G} \), as defined in \eqref{eq:-ProH-Definizione-di-p}.
If
\[
n \ge \frac{1}{4}(p\theta - 4),
\]
then every $u\in C_c^\infty(\mathcal{G}\setminus\{0\})$ satisfies
\begin{equation}\label{eq:-ProH-hardy-non-pesata-2}
\int_{\mathcal{G}}
\frac{|\nabla_{\mathcal{G}} u|^p}{\rho^{p(\theta-1)}}\,dz\,dt
\ge
\left(\frac{n}{n+1}\right)^p
\left|\frac{Q - p\theta}{p}\right|^p
\int_{\mathcal{G}}
\frac{|u|^p}{\rho^{p\theta}}\,dz\,dt.
\end{equation}
\end{corollary}

Two elementary remarks related to the corollary are the following.
The dependence on $N$ in the constant of \eqref{eq:-ProH-hardy-non-pesata-2}
occurs only through the homogeneous dimension
$Q = 2N(n+1) = N Q_n$, where $Q_n = 2n+2$ is the homogeneous
dimension of $\mathbb{H}^n$. Hence the constant takes the form
\[
\left(\frac{n}{n+1}\right)^p \left|\frac{N Q_n - p\theta}{p}\right|^p.
\]

The condition $n \ge \frac14 (p\theta - 4)$ is void whenever $0 \le p\theta \le 4$, for instance, when $p\theta = 2$.

\medskip

We extend the previous construction to step-two Carnot groups with several vertical directions.
Let $n,h\in\mathbb{N}$ with $n\ge h$, and decompose $\mathbb{R}^{N}=\mathbb{R}^{2n}\times\mathbb{R}^{h}$
with coordinates $(z,t)$.  
Choose real, linearly independent, pairwise commuting skew-symmetric matrices
$B^{(1)},\dots,B^{(h)}\in\mathbb{R}^{2n\times 2n}$.  
After an orthogonal change of coordinates, the matrices take the block form
\[
B^{(j)}=\mathrm{diag}\Bigl(
\begin{pmatrix}
0 & \lambda^{(j)}_{1} \\
-\lambda^{(j)}_{1} & 0
\end{pmatrix},
\dots,
\begin{pmatrix}
0 & \lambda^{(j)}_{n} \\
-\lambda^{(j)}_{n} & 0
\end{pmatrix}
\Bigr),
\qquad
\lambda^{(j)}_{i}\ge0,\ i=1,\dots,n,\ j=1,\dots,h.
\]

Endow $\mathbb{R}^{N}$ with the law
\[
(z,t)\circ(\eta,\tau)
=
\bigl(z+\eta,\ t+\tau+\tfrac12\langle Bz,\eta\rangle\bigr),
\qquad
\langle Bz,\eta\rangle
=
(\langle B^{(1)}z,\eta\rangle,\dots,\langle B^{(h)}z,\eta\rangle),
\]
and with dilations $\delta_{\gamma}(z,t)=(\gamma z,\gamma^{2}t)$.
The resulting space $\mathcal{G}=(\mathbb{R}^{N},\circ,\delta_{\gamma})$ is a step-two Carnot group
with horizontal and vertical layers of dimensions $2n$ and $h$.  
The left-invariant horizontal vector fields are
\[
X_{2i-1}
=
\partial_{z_{2i-1}}
+\sum_{j=1}^{h}\frac{\lambda^{(j)}_{i}}{2}\,z_{2i}\,\partial_{t_{j}},
\qquad
X_{2i}
=
\partial_{z_{2i}}
-\sum_{j=1}^{h}\frac{\lambda^{(j)}_{i}}{2}\,z_{2i-1}\,\partial_{t_{j}},
\qquad i=1,\dots,n,
\]
and the only non-vanishing brackets are
\begin{equation}\label{eq:-Gen-braket-non-nulli}
[X_{2i},X_{2i-1}]
=
\sum_{j=1}^{h}\lambda^{(j)}_{i}\,\partial_{t_{j}},
\qquad i=1,\dots,n.
\end{equation}

The linear independence of $B^{(1)},\dots,B^{(h)}$ ensures the existence of indices
$i_{1},\dots,i_{h}\in\{1,\dots,n\}$ such that the matrix
\[
A=\begin{pmatrix}
\lambda^{(1)}_{i_1} & \cdots & \lambda^{(h)}_{i_1} \\
\vdots & \ddots & \vdots \\
\lambda^{(1)}_{i_h} & \cdots & \lambda^{(h)}_{i_h}
\end{pmatrix}\in\mathbb{M}_{h\times h}(\mathbb{R})
\]
is invertible.  
From \eqref{eq:-Gen-braket-non-nulli}, each vertical derivative admits the representation
\[
\partial_{t_j}
= \sum_{k=1}^{h} (A^{-1})_{jk}\,[X_{2 i_k},X_{2 i_k -1}],
\qquad j=1,\dots,h.
\]

The infinitesimal generator of dilations on $\mathcal{G}$ is
\[
\mathcal{E}
=
\sum_{i=1}^{n}\bigl(z_{2i-1}\,\partial_{z_{2i-1}}
+ z_{2i}\,\partial_{z_{2i}}\bigr)
+ 2\sum_{j=1}^{h} t_j\,\partial_{t_j},
\]
and can equivalently be written in terms of the horizontal fields as
\begin{equation}\label{eq:-GCG-definizione-di-E}
\mathcal{E}
=
\sum_{i=1}^{n}\bigl(z_{2i-1}\,X_{2i-1}
+ z_{2i}\,X_{2i}\bigr)
+
2\sum_{j,k=1}^{h}
t_j\,(A^{-1})_{jk}\,
[X_{2 i_k}, X_{2 i_k - 1}].
\end{equation}

For later use, we set
\[
\nabla_{2i-1,2i} u
=
(0,\dots,0, X_{2i-1}u, X_{2i}u, 0,\dots,0), 
\qquad
\nabla_{2i-1,2i}^{\perp}u
=
(0,\dots,0, -X_{2i}u, X_{2i-1}u, 0,\dots,0),
\]
and
\[
\mathbf{z}^{(i)}
=
(0,\dots,0, z_{2i-1}, z_{2i}, 0,\dots,0).
\]

\begin{theorem}\label{thm:caso-carnot-generale}
Let $d\in C^\infty(\mathcal{G}\setminus\{0\})$ be a smooth homogeneous norm, let $p\ge2$, and let $\theta\in\mathbb{R}$.  
Set
\begin{equation}\label{eq:-GCG-definizione-di-Z}
Z_d
=
\frac{z}{d}
+
\sum_{k=1}^{h}\frac{\mathbf{z}^{(i_k)}}{d}
-
2p\theta
\sum_{j,k=1}^{h}
(A^{-1})_{jk}\,
\frac{t_j}{d^{2}}\,
\nabla^{\perp}_{2i_k-1,\,2i_k} d.
\end{equation}
Then, for every $u\in C_c^\infty(\mathcal{G}\setminus\{0\})$,
\begin{equation}\label{eq:-GCG-Hardy-lungo-Z}
\int_{\mathcal{G}}
\frac{|\langle\nabla_{\mathcal{G}}u, Z_d\rangle |^{p}}{d^{p(\theta-1)}}\,dz\,dt
\;\ge\;
\left|\frac{Q-p\theta}{p}\right|^{p}
\int_{\mathcal{G}}
\frac{|u|^{p}}{d^{p\theta}}\,dz\,dt.
\end{equation}
\end{theorem}

The quantity $|Z_d|$ is uniformly bounded for every smooth homogeneous norm $d$.
Computing the precise value of $\max |Z_d|^{2}$ is algebraically involved, and even for simple norms no closed form seems to be available.

\begin{proof}
As in the proof of Theorem~\ref{thm:-Hardy-lungo-Z_d}, the inequality
\eqref{eq:-GCG-Hardy-lungo-Z} follows once
\begin{equation}\label{eq:-GCG-relazione-Z-e-E}
\int_{\mathcal{G}} \frac{|u|^{p-2}u}{d^{p\theta-1}}\,\langle
\nabla_{\mathcal{G}}u, Z_d \rangle\,dz\,dt
=
\int_{\mathcal{G}} \frac{|u|^{p-2}u\,\mathcal{E}u}{d^{p\theta}}\,dz\,dt
\end{equation}
has been established.

An integration by parts yields
\[
2\int_{\mathcal{G}} (A^{-1})_{jk}\,
\frac{|u|^{p-2}u\,t_j}{d^{p\theta}}
[X_{2i_k},X_{2i_k-1}]u\,dz\,dt
=
2\int_{\mathcal{G}} |u|^{p-2}u\,
\left\langle \nabla_{\mathcal{G}}u,\,
(A^{-1})_{jk}\,
\nabla^\perp_{2i_k-1,2i_k}\!\left(\frac{t_j}{d^{p\theta}}\right)\right\rangle dz\,dt,
\]
\[
=\int_{\mathcal{G}} \frac{|u|^{p-2}u}{d^{p\theta-1}}\,\left\langle
\nabla_{\mathcal{G}}u,\,
\lambda^{(j)}_{i_k}(A^{-1})_{jk}\frac{\mathbf{z}^{(i_k)}}{d}
-
2p\theta\,\frac{t_j}{d^{2}}\,
\nabla^\perp_{2i_k-1,2i_k}d
\right\rangle \,dz\,dt.
\]

Since
\[
\sum_{j=1}^h \lambda^{(j)}_{i_k}(A^{-1})_{jk}=1,
\]
summing in $k$ and using the explicit expression of $\mathcal{E}$ in
\eqref{eq:-GCG-definizione-di-E} proves
\eqref{eq:-GCG-relazione-Z-e-E}.
\end{proof}

\medskip
\noindent{\textbf{Acknowledgments}.}
L. Fanelli is partially supported 
by the Spanish Agencia Estatal de Investigaci\'{o}n
through BCAM Severo Ochoa accreditation CEX2021-001142-S/MCIN/AEI/10.13039/501100011033 and 
CNS2023-143893, by IKERBASQUE and by the research project PID2024-155550NB-100 funded by MICIU/AEI/10.13039/501100011033 and FEDER/EU.

\bibliographystyle{plain} 
\bibliography{bibliography-v3}
\end{document}